\documentclass[a4paper,reqno,right =0cm,left=0cm]{amsart}

\usepackage[utf8]{inputenc}
\usepackage[T1]{fontenc}

\usepackage{amsmath}
\usepackage{amssymb}
\usepackage{float}
\usepackage{graphicx}
\usepackage{xcolor}
\usepackage{colortbl}
\usepackage{mathtools}
\usepackage{mathrsfs}
\usepackage{textcomp}
\usepackage{hhline}
\usepackage{comment}
\usepackage{mathdots,enumerate}

\newtheorem{theorem}{Theorem}[section]
\newtheorem{proposition}[theorem]{Proposition}
\newtheorem{corollary}[theorem]{Corollary}
\newtheorem{lemma}[theorem]{Lemma}

\theoremstyle{definition}
\newtheorem{definition}[theorem]{Definition}
\newtheorem{example}[theorem]{Example}

\newtheorem{remark}[theorem]{Remark}

\numberwithin{equation}{section}
\numberwithin{figure}{section}

\usepackage{tikz-cd}
\usetikzlibrary{positioning,arrows,decorations.markings}
\tikzset{negated/.style={
        decoration={markings,
            mark= at position 0.5 with {
                \node[transform shape] (tempnode) {$\backslash\!\!\backslash$};
            }
        },
        postaction={decorate}
    }
}

\def\e{{\epsilon}}

\newcommand{\K}{\mathbb{K}}
\newcommand{\R}{\mathbb{R}}
\newcommand{\C}{\mathbb{C}}
\newcommand{\N}{\mathbb{N}}
\newcommand{\Dc}{\mathcal{D}}
\newcommand{\Hc}{\mathcal{H}}
\newcommand{\Lc}{\mathcal{L}}
\newcommand{\cM}{\mathcal{M}}

\newcommand{\Vc}{\mathcal{V}}

\DeclareMathOperator{\re}{Re}
\DeclareMathOperator{\diag}{diag}
\DeclareMathOperator{\im}{Im}
\DeclareMathOperator{\rk}{rk}
\DeclareMathOperator{\gr}{gr}
\DeclareMathOperator{\dom}{dom}
\DeclareMathOperator{\ran}{ran}
\DeclareMathOperator{\Span}{span}
\DeclareMathOperator{\mul}{mul}

\newenvironment{smallpmatrix}
{\left(\begin{smallmatrix}}
{\end{smallmatrix}\right)}

\newenvironment{smallbmatrix}
{\left[\begin{smallmatrix}}
{\end{smallmatrix}\right]}

\title{A~linear relation approach to port-Hamiltonian differential-algebraic equations}

\author[H.~Gernandt]{Hannes Gernandt}
\address{Institut f\"{u}r Mathematik \\
             Technische Universit\"{a}t Ilmenau \\
             98693 Ilmenau\\
             Germany}
\email{hannes.gernandt@tu-ilmenau.de}

\author[F.E.~Haller]{Fr\'ed\'eric Enrico Haller}
\address{Universit\"{a}t Hamburg, Bundesstra\ss e 55, 20146 Hamburg, Germany}
\email{frederic.haller@uni-hamburg.de}

\author[T.~Reis]{Timo Reis}
\address{Universit\"{a}t Hamburg, Bundesstra\ss e 55, 20146 Hamburg, Germany}
\email{timo.reis@uni-hamburg.de}

 \thanks{This
       work was supported by the grants RE 2917/4-1 and WO 2056/1-1 ``Systems theory of partial differential-algebraic equations'' by the Deutsche Forschungsgemeinschaft (DFG)}

\begin{document}

\begin{abstract}
We consider linear port-Hamiltonian differential-algebraic equations (pH-DAEs).
Inspired by the geometric approach of {\sc Maschke} and {\sc van der Schaft} \cite{MascvdSc18} and
the linear algebraic approach {\sc Mehl, Mehrmann} and {\sc Wojtylak} \cite{MehlMehrWojt18}, we present another view by using the theory of {\em linear relations}. We show that this allows to elaborate the differences and mutualities of the geometric and linear algebraic views, and we introduce a~class of DAEs which comprises these two approaches. We further study the properties of matrix pencils arising from our approach via linear relations.
\end{abstract}



\maketitle

\pagestyle{myheadings}
\thispagestyle{plain}
\markboth{\textsc{H. Gernandt, F. E. Haller, and T. Reis}}{\textsc{Linear relations and port-Hamiltonian DAEs}}

\section{Introduction}\label{sec:intro}

Port-Hamiltonian modelling provides a framework allowing for a systematic port-based network modelling of complex lumped parameter systems from various physical domains. This modelling is based on energy considerations of individual systems and their interconnection. In the past decades, this approach has gained particularly increased attention from different communities, such as {geometric mechanics} and {mathematical systems theory}, from which different definitions of port-Hamiltonian systems emerged, see \cite{JacoZwar12, JvdS14,vdSc13} for an overview. 

This article is devoted to the analysis and comparison of two approaches to port-Hamiltonian differential-algebraic equations (DAEs).
One approach by {\sc Mehl}, {\sc Mehrmann} and {\sc Wojtylak} in \cite{MehlMehrWojt18} is of linear algebraic nature, and is based on the study of the class
\begin{align}
\label{DAE}
\tfrac{\rm d}{{\rm d}t}E z(t) = Az(t),
\end{align}
with, for $\K\in\{\R,\C\}$, $E,Q\in\K^{n\times m}$ and $D\in\K^{n\times n}$,
\begin{align}
\label{phintro}
A=DQ,\quad  Q^*E=E^*Q,\quad \text{and}\quad  { D+D^*\leq 0,}
\end{align}
where $M\geq0$ ($M\leq0$) refers to symmetry and positive (negative) semi-definiteness of the square matrix $M$,
and the property $D+D^*\leq 0$ is called \textit{dissipativity} of $D$. Note that \cite{MehlMehrWojt18} uses the notation $D=J-R$ for  $J,R\in\K^{n\times n}$ with $J$ skew-Hermitian and $R\geq0$, and we stress that a~matrix is dissipative if, and only if, it can be represented as such a~matrix difference as above.\\
Special emphasis is placed on the case where
\begin{align*}
Q^*E\geq 0,
\end{align*}
since, oftentimes, $\frac12z(t)^*Q^*Ez(t)$ corresponds to the physical energy of the system \eqref{DAE} at time $t$ \cite[Ex.~1]{MehlMehrWojt18}.
The properties \eqref{phintro} allow a~deep analysis of the Kronecker structure and location of eigenvalues of matrix pencils $sE-DQ\in\K[s]^{n\times m}$ and, consequently, an~understanding of the qualitative solution behavior of \eqref{DAE} \cite{MehlMehrWojt18}.\\
Another approach to port-Hamiltonian DAEs by {\sc Maschke} and {\sc van der Schaft} \cite{MascvdSc18} is of geometric nature. Such systems are specified by the relation
\begin{align}
\label{MSdefintro}
(e(t),\tfrac{\rm d}{{\rm d}t} x(t))\in \Dc,\quad  (x(t),e(t))\in\Lc
\end{align}
for some $\K^n$-valued function $e(\cdot)$, where  $\Lc$ and $\Dc$ are the so-called {\em Lagrangian} and {\em Dirac subspaces} of $\K^{2n}$, see Section \ref{sec:relations}. Note that, in \cite{MascvdSc18}, the first inclusions in \eqref{MSdefintro} is actually written as $(-\tfrac{\rm d}{{\rm d}t} x(t),e(t))\in\Dc$. However, it can be shown that this is equivalent to $(e(t),\tfrac{\rm d}{{\rm d}t}x(t))\in\widetilde{\Dc}$, for some alternative Dirac subspaces $\widetilde{\Dc}$.
It is shown in \cite{MascvdSc18} that Dirac and Lagrange subspaces admit kernel and image representations $\Dc=\ker[K,L]=\ran\left[\begin{smallmatrix}L^*\\K^*\end{smallmatrix}\right]$ and $\Lc=\ran\left[\begin{smallmatrix}P\\S\end{smallmatrix}\right]=\ker[S^*,-P^*]$ for some $K,P,L,S\in \K^{n\times n}$ with $KL^*=-LK^*$, $S^*P=P^*S$
and
$\rk[\,K,\, L\,]=\rk[\,P,\, S\,]=n$. This allows, by taking $\left(\begin{smallmatrix}e(t)\\z(t)\end{smallmatrix}\right)=\left[\begin{smallmatrix}P\\S\end{smallmatrix}\right]x(t)$, to rewrite \eqref{MSdefintro} as a DAE $L\tfrac{\rm d}{{\rm d}t} Px(t)=-KSx(t)$.\\
The purpose of this article is to present the relation between these two approaches. To this end, we present another view via so-called {\em linear relations}, a~concept which has been treated in several textbooks \cite{BehrHassdeSn20,Cross}. Via linear relations, we present a~class which comprises both the linear algebraic and geometric approach. In particular, we make use of three facts:
\begin{table}[H]
\begin{center}\fbox{
  \begin{minipage}[t]{10cm}
~\\[-4mm]
\begin{enumerate}[(i)]
\item the geometric concept of Dirac structure translates to the notion of {\em skew-adjoint linear relation} in the language of linear relations,\\[-2mm]
\phantom{x}
  \item Lagrangian subspaces correspond to {\em self-adjoint linear relations}, and\\[-2mm]
\phantom{x}
  \item dissipative matrices can be generalized to {\em dissipative linear relations}.\\[-2mm]
\phantom{x}
\end{enumerate}
\end{minipage}
}\caption{Linear relations and Dirac/Lagrange subspaces}\label{fig:differencesDLL}
\end{center}\end{table}

We will see that \eqref{MSdefintro} can be written, in the language of linear relations, as
\begin{align}
\label{ourdefintro}
(x(t),\dot x(t))\in \Dc\Lc,
\end{align}
where $\Dc \Lc$ is the product of the linear relations $\Dc$ and $\Lc$, see Section~\ref{sec:relations}. By choosing matrices $E,A\in\K^{n\times q}$ with
\begin{equation}
\boxed{\begin{aligned}~\\[-4mm]
\qquad\Dc\Lc=\ran\begin{bmatrix}E\\ A\end{bmatrix},\qquad\\[-4mm]
\phantom{x}\end{aligned}
}\label{factor3}
\end{equation}
the differential inclusion \eqref{ourdefintro} can be transformed to the DAE 
\[
    \begin{smallpmatrix} x(t)\\\dot{x}(t)\end{smallpmatrix}=\begin{smallbmatrix}E\\ A\end{smallbmatrix}z(t).
\]
which has to be solved for $x(\cdot)$ and some $\K^q$-valued function $z(\cdot)$. 
It can be seen that an elimination of $x(\cdot)$ leads to $\tfrac{\rm d}{{\rm d}t}E z(t) = Az(t)$. On the other hand, it can be shown that for 
matrices with properties as in \eqref{phintro} and choosing
$\Dc=\ran\begin{smallbmatrix}I\\D\end{smallbmatrix}$, $\Lc=\ran\begin{smallpmatrix}E\\Q\end{smallpmatrix}$,
the equations \eqref{DAE} and \eqref{ourdefintro} are equivalent. Hereby, we will see that $\Dc$ is a~so-called {\em dissipative relation} and $\Lc$ is a~{\em symmetric relation}. These are concepts which are slightly more general than skew-adjoint and self-adjoint relations.

These findings allow a~comparison of the approaches in
\cite{MascvdSc18} and \cite{MehlMehrWojt18}:
Namely, to analyze whether a~given pH-DAE in the sense of \cite{MehlMehrWojt18} is one in the sense of \cite{MascvdSc18}, it has to be investigated whether the linear relation $\Lc=\ran\begin{smallpmatrix}E\\Q\end{smallpmatrix}$ is self-adjoint subspace $\Lc$ and a~skew-adjoint subspace $\Dc$. 
 On the other hand, to analyze whether a~pH-DAE which in the sense of \cite{MascvdSc18} is one in the sense of \cite{MehlMehrWojt18}, it has to be investigated whether $\Dc=\gr D$ for some dissipative matrix $D\in\K^{n\times n}$, where $\gr D$ stands for the {\em graph of $D$}, i.e., $\gr D=\ran\begin{smallbmatrix}I\\ D\end{smallbmatrix}$. Moreover, a~joint structure of both approaches are DAEs $\tfrac{\rm d}{{\rm d}t}E z(t) = Az(t)$ for which \eqref{factor3} holds for some dissipative relation $\Dc$ symmetric relation $\Lc$.

Besides a~comparison of both existing approaches to pH-DAEs, we will investigate structural properties of DAEs belonging to the aforementioned joint structure, such as an analysis of the Kronecker structure of the pencil $sE-A$ with \eqref{factor3} with $\Dc$ and $\Lc$ being dissipative and symmetric, respectively. Sometimes we will impose the additional assumption that $\Lc$ is a {\em nonnegative linear relation}, which generalizes the condition that $E^*Q$ is positive semi-definite. Note that the latter is motivated by
quadratic form $\frac12x(t)^*Q^*Ex(t)$ oftentimes standing for physical energy of the system at time $t$.

\begin{figure}[H]
{\footnotesize
\begin{center}
\resizebox{0.98\linewidth}{!}{  \begin{tikzpicture}[thick,node distance = 22ex,implies/.style={double,double equal sign distance,-implies}, box/.style={fill=white,rectangle, draw=black},
  blackdot/.style={inner sep = 0, minimum size=3pt,shape=circle,fill,draw=black},plus/.style={fill=white,circle,inner sep = 0,very thick,draw},
  metabox/.style={inner sep = 3ex,rectangle,draw,dotted,fill=gray!20!white}]

    \node (ms)     [metabox,minimum size=1.4cm,]  {\parbox{4.6cm}{\rm $\begin{matrix}\ran\begin{smallbmatrix}E\\ A\end{smallbmatrix}=\Dc\Lc,\\ \Dc\,{\rm Dirac\, and}\, \Lc\;{\rm Lagrangian}\\ \textsc{(Maschke, van der Schaft)}\end{matrix}$}};

    \node (dms)     [right of = ms,metabox,minimum size=1.4cm,xshift=40ex]  {\parbox{4.6cm}{\rm $\begin{matrix}\ran\begin{smallbmatrix}E\\ A\end{smallbmatrix}=\Dc\Lc,\, \Dc\; {\rm maximally}\\ {\rm dissipative, and }\;\Lc\;{\rm self{-}adjoint}\end{matrix}$}};

    \node (mmw)     [below of = ms,metabox,minimum size=1.4cm]  {\parbox{4.6cm}{\rm $\begin{matrix}\ran\begin{smallbmatrix}E\\ A\end{smallbmatrix}=(\gr D) \Lc,\\ D+D^*\leq 0,\,\Lc  \,{\rm symmetric}\\ \textsc{(Mehl, Mehrmann, Wojtylak)}\end{matrix}$}};

    \node (nonmax)     [below of = dms,metabox,minimum size=1.4cm]  {\parbox{4.6cm}{\rm $\begin{matrix}\ran\begin{smallbmatrix}E\\ A\end{smallbmatrix}=\Dc \Lc,\\ \Dc\;{\rm dissipative, and }\; \Lc\;{\rm symmetric}\end{matrix}$}};

 \node (newmmw) [below of = ms,yshift=4ex,xshift=21ex] {};

 \node (newnonmax) [below of = dms,yshift=4ex,xshift=-21ex] {};

 \node (newms) [right of = ms,yshift=-6.2ex,xshift=-15ex] {};

 \node (newmmw2) [right of = mmw,yshift=6ex,xshift=-15ex] {};

    \draw[->,semithick,double,double equal sign distance,>=stealth,negated] (newms)  -- node [right] {\,\,{\rm Example\,\ref{ex:notmmw}} }  (newmmw2);

    \draw[->,semithick,double,double equal sign distance,>=stealth,negated] (mmw)  -- node [left] {{\rm Examples\,\ref{ex:ex1}\&\ref{ex:notms}} \,\,}  (ms);

    \draw[->,semithick,double,double equal sign distance,>=stealth] (ms)  -- node [above] {}  (dms);


    \draw[->,semithick,double,double equal sign distance,>=stealth] (mmw)  -- node [above] {}  (nonmax);


    \draw[->,semithick,double,double equal sign distance,>=stealth] (dms)  -- node [left,yshift=-2ex,,xshift=3ex] {\,}  (nonmax);

  \end{tikzpicture}}

\end{center}}

  \caption{Relations between geometric concepts and those from the theory of linear relations}
  \end{figure}
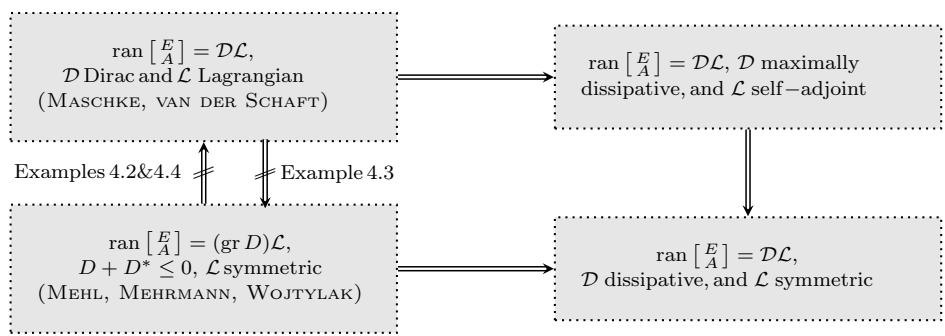




Note that both the approaches in \cite{MascvdSc18} and \cite{MehlMehrWojt18} allow the incorporation of further {\em external variables}, such as inputs and outputs. In this article we will restrict to the uncontrolled case for sake of better overview.

The paper is organized as follows: in Section \ref{sec:pencils} we recall basic facts on matrix pencils, such as the Kronecker form. In Section \ref{sec:relations} the basic notions from the theory of linear relations and properties of dissipative, nonnegative and self-adjoint subspaces are presented. This can be used in Section~\ref{sec:comp} for a~port-Hamiltonian formulation via linear relations, along with a~detailed comparison of the approaches of {\sc Mehrmann}, {\sc Mehl} and {\sc Wojtylak} and the formulation \eqref{MSdefintro} by {\sc Maschke} and {\sc van der Schaft} via Dirac and Lagrange structures. By using linear relations, we will introduce a~novel class which can be seen as a {\em least common multiple} of both existing approaches.
Section \ref{sec:reg} is devoted to the characterization of regularity of the pencils arising in this novel class, and, in Section~\ref{sec:main} we use, the additional assumption that the linear relation $\Lc$ in \eqref{MSdefintro} is nonnegative and perform a~structural analysis of such systems. In particular, we analyze the index and the location of the eigenvalues of the underlying matrix pencil.  




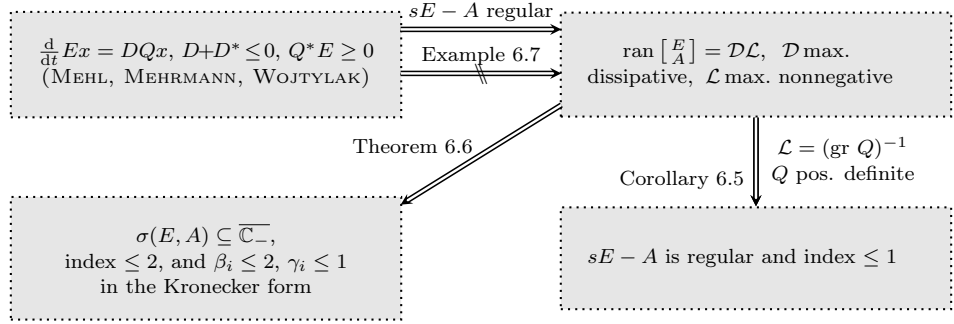
\begin{figure}
{\footnotesize
\begin{center}
  \resizebox{0.98\linewidth}{!}{\begin{tikzpicture}[thick,node distance = 22ex,implies/.style={double,double equal sign distance,-implies}, box/.style={fill=white,rectangle, draw=black},
  blackdot/.style={inner sep = 0, minimum size=3pt,shape=circle,fill,draw=black},plus/.style={fill=white,circle,inner sep = 0,very thick,draw},
  metabox/.style={inner sep = 3ex,rectangle,draw,dotted,fill=gray!20!white}]

    \node (mm)     [metabox,minimum size=1.4cm]  {\parbox{4.6cm}{\rm $\begin{matrix}\tfrac{\rm d}{{\rm d}t}Ex=DQx,\, D\!\!+\!\!D^*\!\leq \!0,\, Q^*E\geq 0\\ \textsc{(Mehl, Mehrmann, Wojtylak)}\end{matrix}$}};
    \node (ms)     [right of = mm,metabox,minimum size=1.4cm,xshift=40ex]  {\parbox{4.6cm}{\rm $\begin{matrix}\ran\begin{smallbmatrix}E\\ A\end{smallbmatrix}=\Dc\Lc,\;\; \Dc\,{\rm max.\ }\\ {\rm dissipative, }\;\; \Lc\,{\rm max.\ nonnegative}\end{matrix}$}};

    \node (kron) [below of = mm, metabox,minimum size=1.4cm]  {\parbox{4.6cm}{\rm\centering $\sigma(E,A)\subseteq\overline{\C_-}$,\\  index $\leq 2$, and $\beta_i\leq 2$,  $\gamma_i\leq 1$\\in the Kronecker form}};
    \node (ind) [below of = ms, metabox,minimum size=1.4cm]  {\parbox{4.6cm}{\rm $sE-A$ is regular and index $\leq 1$}};

\node (test4) [below of =mm, xshift=21ex,yshift=26ex]{};
\node (test5) [below of =mm, xshift=21ex,yshift=21ex]{};
\node (test6) [below of =mm, xshift=41ex,yshift=26ex]{};
\node (test7) [below of =mm, xshift=41ex,yshift=21ex]{};
\node (test8) [below of =mm, xshift=21ex,yshift=5.5ex]{};

\node (newms) [right of =mmw, xshift=19ex,yshift=18ex]{};

    \draw[->,semithick,double,double equal sign distance,>=stealth] (ms)  -- node [left,yshift=-2ex] {Corollary \ref{cor:index1}\,}  (ind);

    \draw[->,semithick,double,double equal sign distance,>=stealth] (ms)  -- node [right,yshift=0ex] {\,\,$\begin{matrix} \text{$\Lc=(\gr \,Q)^{-1}$}\\ \text{$Q$ pos. definite}\end{matrix}$}  (ind);


    \draw[->,semithick,double,double equal sign distance,>=stealth] (test4)  -- node [above] {$sE-A$ regular}  (test6);

    \draw[->,semithick,double,double equal sign distance,>=stealth] (newms)  -- node [left,yshift=1ex] {Theorem \ref{thm:singular}} (test8);

\draw[->,semithick,double,double equal sign distance,>=stealth,negated] (test5)  -- node [above] {Example \ref{ex:MMW}} (test7);

  \end{tikzpicture}}

\end{center}}

  \caption{Properties of matrix pencils arising in port-Hamiltonian formulations.}
  \end{figure}

\section{Preliminaries on matrix pencils}
\label{sec:pencils}
The analysis of DAEs of the form \eqref{DAE} leads to the study of \textit{matrix pencils}, which are first-order matrix polynomials  $sE-A\in\K[s]^{n\times m}$ with coefficient matrices $E,A\in\K^{n\times m}$. To this end, note that $\K[s]$ denotes the ring of polynomials over $\K$, and $\K(s)$ is the quotient field of $\K[s]$.

First, we recall the \textit{Kronecker form} for matrix pencils, see e.g.\ \cite[Chap.~XII]{Gant59}, i.e.\ there exist invertible matrices $S\in\K^{n\times n}$ and $T\in\K^{m\times m}$ with
\begin{align}
\label{KCF}
S(sE-A)T=\begin{bmatrix} sI_{n_0}-J& 0&0&0 \\0& sN_\alpha-I_{|\alpha|}&0&0\\ 0&0& sK_{\beta}-L_{\beta} &0\\ 0&0&0& sK_{\gamma}^\top-L_{\gamma}^\top\end{bmatrix}
\end{align}
with $J$ in Jordan canonical form over $\K$, see e.g.\  \cite[Secs.~3.1\& 3.4]{HornJohn13} and, for multi-indices
$\alpha=(\alpha_i)_{i=1,\ldots,\ell_\alpha}$,
$\beta=(\beta_i)_{i=1,\ldots,\ell_\beta}$,
$\gamma=(\gamma_i)_{i=1,\ldots,\ell_\gamma}$,
\[N_\alpha=\diag(N_{\alpha_i})_{i=1,\ldots,\ell_\alpha},\;\;
K_\beta=\diag(K_{\beta_i})_{i=1,\ldots,\ell_\beta}
,\;\;
L_\gamma=\diag(L_{\gamma_i})_{i=1,\ldots,\ell_\gamma},
\]
where, for $k\in \N$ with $k\geq1$, $N_k$ is a~nilpotent Jordan block of size $k\times k$, and $K_k\coloneqq[I_{k-1},0]\in\R^{(k-1)\times k}$, $L_k=[0,I_{k-1}]\in\R^{(k-1)\times k}$.
The numbers $\alpha_i$ for $i=1,\ldots,\ell_\alpha$ are referred to as \textit{sizes of the Jordan blocks at $\infty$}, whereas for $i=1,\ldots,\ell_\beta$, $j=1,\ldots,\ell_\gamma$, the numbers $\beta_i-1$ and $\gamma_j-1$ are respectively called \textit{column} and \textit{row minimal indices}, and are well-defined by $sE-A$. Furthermore, we can define the \textit{(Kronecker) index} $\nu$ of the DAE \eqref{DAE} based on the Kronecker canonical form \eqref{KCF} as
\begin{align}
\label{def:index}
\nu=\max\{\alpha_1,\ldots,\alpha_{\ell_\alpha},\gamma_1,\ldots,\gamma_{\ell_\gamma},0\}.
\end{align}
In this sense a~DAE \eqref{DAE} has index one if $N_{\alpha}=0$ and if the fourth block column in \eqref{KCF} is zero. The upper left subpencil 
$\diag(sI_{n_0}-J, sN_\alpha-I_{|\alpha|})$
in \eqref{KCF} is called the \textit{regular part} of the Kronecker form \eqref{KCF}. A~number $\lambda\in \C$ is an {\em eigenvalue of the pencil $sE-A$}, if
 $\rk_\C\lambda E-A<\rk_{\K(s)}sE-A$, and we write
\[\sigma(E,A)\coloneqq\{\lambda\in\C\,|\,\lambda\text{ is an eigenvalue of $sE-A$}\}.\]
Note that $\lambda \in\C$ is an eigenvalue of the pencil $sE-A$ if, and only if, $\lambda$ is an eigenvalue of the matrix $J$ in the Kronecker form \eqref{KCF}. An eigenvalue $\lambda\in\sigma(E,A)$ is called \emph{semi-simple} if $J$ in \eqref{KCF} has no Jordan blocks of size greater or equal to two at $\lambda$. Note that semi-simplicity is well-defined, i.e., it does not depend on the given Kronecker form of $sE-A$.



A~square pencil $sE-A\in\K[s]^{n\times n}$ is called \textit{regular}, if $\det(sE-A)$ is not the zero polynomial. This is equivalent to the property that $sE-A$ has no row and column minimal indices. The Kronecker form of a~regular pencil is also called \textit{Weierstra\ss\ form}. For regular matrix pencils, set of eigenvalues fulfills
\[
\sigma(E,A)=\{\lambda\in\C \,|\, \det(\lambda E-A)=0\}.
\]
Note that regularity implies that $sE-A$ is invertible as a~matrix with entries in $\K(s)$. In this case, $\sigma(E,A)$ coincides with the set of poles of $(sE-A)^{-1}\in\K(s)^{n\times n}$.
%

We state another elementary lemma which can be derived directly from the Weierstra\ss\ canonical form for regular matrix pencils. We will characterize the index by means of the growth of the {\em resolvent} $(sE-A)^{-1}$ on a~real half-axis. To this end, we will use a~certain matrix norm. Note that, by finite-dimensionality of the systems, the result is independent of concrete choice of the matrix norm.

\begin{lemma}
\label{lem:elementary}
Let the pencil $sE-A\in\K[s]^{n\times n}$ be regular. Then the index of $sE-A$ is equal to the smallest number $k$ for which there exists some $M>0$ and $\omega\in\R$, such that
\begin{align*}
\forall \lambda>\omega:\quad\|(\lambda E-A)^{-1}\|\leq M|\lambda|^{k-1}.
\end{align*}
Moreover, the size of the largest Jordan block at an eigenvalue $\lambda$ of $sE-A$ is equal to the order of $\lambda$ as a pole of $(sE-A)^{-1}\in\K(s)^{n\times n}$.
\end{lemma}

\begin{definition}
\label{def:posreal}
A~matrix $G(s)\in\K(s)^{n\times n}$ is called {\em positive real}, if
\begin{itemize}
\item[\rm (a)] $G(s)$ has no poles in the open right complex half-plane.
\item[\rm (b)]$G(\lambda)+G(\lambda)^*\geq 0$ for all $\lambda\in\C$ with $\re\lambda>0$.
\end{itemize}
\end{definition}
It can be immediately seen that a matrix pencil $sE-A\in\K[s]^{n\times n}$ is {positive real} if, and only if, $E=E^*\geq 0$ and $A+A^*\leq 0$. We recall some properties of positive real matrix pencils, which can be immediately concluded by a~combination of \cite[Lem.~2.6]{BergReis14} with \cite[Cor.~2.3]{BergReis14b}. 
\begin{lemma}
\label{lem:posreal}
Let $sE-A\in\K[s]^{n\times n}$ be a~positive real pencil. Then the following holds.
\begin{itemize}
    \item[\rm (a)] $sE-A$ is regular if, and only if, $\ker E\cap\ker A=\{0\}$.
    \item[\rm (b)] The row and column minimal indices are at most zero and their numbers coincide.
    \item[\rm (c)] The eigenvalues of the pencil are contained in the closed left half-plane $\overline{\C_-}$ and the eigenvalues on the imaginary axis are semi-simple.
    \item[\rm (d)] The index of $sE-A$ is at most two.
\end{itemize}
\end{lemma}

\section{Preliminaries on linear relations}
\label{sec:relations}
We will introduce the notion of \textit{linear relation} on $\K^n$, which are basically subspaces of $\K^n\times\K^n\cong\K^{2n}$. An introduction to linear relations can be found e.g.\ in \cite{BehrHassdeSn20, Cross}. Throughout this article, we assume that $\K^n$ is equipped with the standard scalar product $\langle\cdot,\cdot\rangle: (x,y)\mapsto y^*x$. An important special case of a~linear relation is the graph of a square matrix $M\in\K^{n\times n}$, i.e.
\[\gr M\coloneqq\{(x,Mx) \,|\, x\in\K^n\}.\]
This motivates to define the following concepts for linear relations. Note that, by writing $(x,y)\in\K^{2n}$, we particularly mean that $x,y\in\K^n$.
\begin{definition}[Concepts and operations on linear relations]
Let $n\in\N$, and $\Lc,\cM\subset\K^{2n}$ be linear relations in $\K^n$.\\
The  {\em domain, kernel, range} and {\em multi-valued part} are
\begin{align*}
    \dom \cM&\coloneqq\{x\in \K^n \,|\, (x,y)\in \cM\},\quad
    &\ker \cM&\coloneqq\{x\in \K^n \,|\, (x,0)\in \cM\},\\
    \ran \cM&\coloneqq\{y\in \K^n \,|\, (x,y)\in \cM\},\quad
    &\mul \cM&\coloneqq\{y\in \K^n \,|\, (0,y)\in \cM\},
\end{align*}
and \textit{scalar multiplication with $\alpha\in\K$}, \textit{operator-like sum}, {\em product}, {\em inverse} and {\em adjoint}
are defined by
\begin{align*}
    \alpha\cM&\coloneqq\{(x,\alpha y)\in\K^{2n} \,|\, (x,y)\in\cM\},\\
    \Lc+\cM&\coloneqq\{(x,y_1+y_2)\in\K^{2n} \,|\, (x,y_1)\in\Lc,(x,y_2)\in\cM\},\\
    \cM\Lc&\coloneqq\{ (x,z)\in\K^{2n} \,|\, \text{$\exists y\in\Hc$ s.t.\ } (x,y)\in\Lc, (y,z)\in\cM \},\\
    \cM^{-1}&\coloneqq\{(y,x)\in\K^{2n} \,|\, (x,y)\in\cM\},\\
    \cM^*&\coloneqq\{(x,y)\in\K^{2n} \,|\, \langle w,x\rangle=\langle v,y\rangle \;\;\forall \,(v,w)\in\cM\}.
\end{align*}
A~linear relation with $\cM\subseteq \cM^*$ is called {\em symmetric}, whereas $\cM$ is {\em self-adjoint}, if $\cM=\cM^*$. Likewise, $\cM$ with
$\cM\subseteq -\cM^*$ is called {\em skew-symmetric}, and  $\cM$ is {\em skew-adjoint}, if it has the property $\cM=-\cM^*$.
\end{definition}

If $\K=\C$ then a~linear relation $\cM$ is symmetric (self-adjoint) if, and only if, $\imath \cM$ is skew-symmetric (skew-adjoint), where $\imath$  denotes the imaginary unit.

Note that the operator-like sum of two linear relations $\Lc,\cM\subset\K^{2n}$ is \underline{not} the \textit{componentwise sum}, which is defined by
\begin{align*}
    \Lc\widehat+\cM\coloneqq\{(x_1+x_2,y_1+y_2)\in \K^{2n} \,\,|\,\, (x_1,y_1)\in\Lc, (x_2,y_2)\in\cM\}.
\end{align*}
If $\Lc$ and $\cM$ satisfy
$\Lc\cap\cM=\{0\}$ we will write $\Lc\widehat{\dot+}\cM$ for the componentwise sum of $\Lc$ and $\cM$.
We oftentimes use the identity
\begin{align}
\label{identity}
(-\cM^*)^{-1}=\cM^{\perp},
\end{align}
where $\cM^{\perp}$ is the orthogonal complement of $\cM\subseteq\K^{2n}$. In particular, we can conclude that
\begin{align*}
    2n=\dim\cM+\dim\cM^\perp=\dim\cM+\dim(\cM^*)^{-1}
    =\dim\cM+\dim\cM^*,
\end{align*}
which gives
\begin{equation}\label{eq:dimM*}
    \dim\cM^*=2n-\dim\cM.
\end{equation}
We will also use that a linear relation $\cM$ in $\K^n$ can be written as $\cM=\ker[K,L]$ or $\cM=\ran\begin{smallbmatrix}F\\ G\end{smallbmatrix}$ with matrices $F,G\in\K^{n\times l}$ and $K,L\in\K^{l\times n}$ which we will refer to as \emph{kernel} and \emph{image representation}. These representations always exist, see e.g.\ \cite[Thm.~3.3]{BergTrunWink16}, if $\K=\C$, for each choice of $l\in\N$ such that $l\geq\dim\cM$. The proof of the existence of the range representation for $\K=\R$ can also be derived from the above mentioned result.

Together with \eqref{identity} we have for $\cM=\ran\begin{smallbmatrix}F\\ G\end{smallbmatrix}=\ker[K,L]$ that
\begin{align}
\label{Lstar}
\cM^*=\ker[G^*,-F^*]=\ran\begin{smallbmatrix}L^*\\ -K^*\end{smallbmatrix}.
\end{align}

 In literature on port-Hamiltonian systems, self-adjoint linear relations in $\K^{n}$ appear under the name \textit{Lagrangian subspaces}, whereas skew-adjoint linear relations are called \textit{Dirac subspaces}, see e.g.\ \cite{MascvdSc18}.

In the following result we characterize symmetry and self-adjointness of a linear relation by means of certain properties of the matrices in the range and kernel representation.

\begin{lemma}
\label{lem:symequiv}
Let $\cM\subset\K^{2n}$ be a~linear relation. Then $\cM$ is symmetric if, and only if, $\cM=\ran\begin{smallbmatrix}F\\ G\end{smallbmatrix}$ for some $F,G\in\K^{n\times l}$ with $G^*F=F^*G$. Moreover, the following statements are equivalent.
\begin{itemize}
    \item[\rm (a)] $\cM$ is self-adjoint,
    \item[\rm (b)] $\cM$ is symmetric and $\dim\cM=n$,
    \item[\rm (c)] $\cM=\ker[K,L]$ for some $K,L\in\K^{n\times n}$ with $KL^*=LK^*$ and $\rk[K,L]= n$.
\end{itemize}
\end{lemma}
\begin{proof}
To prove the first equivalence, assume that $\cM\subset\K^{2n}$ is symmetric and let $F,G\in\K^{n\times l}$ such that $\cM=\ran \begin{smallbmatrix}F\\ G\end{smallbmatrix}$. The symmetry of $\cM$ together with \eqref{Lstar} now implies that
\[\forall\, z\in\K^n:\;0=[G^*,-F^*]\underbrace{\begin{smallbmatrix}F\\ G\end{smallbmatrix}z}_{\in\cM\subset\cM^*}=(G^*F-F^*G)z,\]
whence $G^*F=F^*G$.\\
Conversely, assume that
$\cM=\ran\begin{smallbmatrix}F\\ G\end{smallbmatrix}$ for some $F,G\in\K^{n\times l}$ with $G^*F=F^*G$. Let $(x_1,y_1),(x_2,y_2)\in\cM$. Then there exists some $z_1,z_2\in\K^n$ with $x_1=Fz_1$, $y=Gz_1$, $x_2=Fz_2$ and $y_2=Gz_2$. Then
\[\langle y_2,x_1\rangle=\langle Gz_2,Fz_1\rangle
=\langle z_2,G^*Fz_1\rangle
=\langle z_2,F^*Gz_1\rangle
=\langle Fz_2,Gz_1\rangle
=\langle x_2,y_1\rangle,\]
i.e., $\cM$ is symmetric. We now show the equivalences (a)-(c). \\
``(a)$\Rightarrow$(b)'': If $\cM\subset\K^{2n}$ is self-adjoint, then, by \eqref{eq:dimM*}, \[\dim \cM=\dim \cM^*=2n-\dim \cM,\] which gives $\dim\cM=n$.\\
``(b)$\Rightarrow$(c)'': Assume that $\cM\subset\K^{2n}$ is symmetric and $\dim\cM=n$. By the first equivalence there exist $F,G\in\K^{n\times n}$ such that $\cM=\ran \begin{smallbmatrix}F\\ G\end{smallbmatrix}$ and $G^*F=F^*G$. Since $\cM=\cM^*$, the choices of $K=G^*$ and $L=-F^*$ together with \eqref{Lstar} lead to $\cM=\ker[K,L]$ with $KL^*=LK^*$. Further, we have
\[n=\dim\cM=\rk\begin{smallbmatrix}F\\ G\end{smallbmatrix}=\rk[K,L].\]
``(c)$\Rightarrow$(a)'': Assume that $\cM=\ker[K,L]$ for
$K,L\in\K^{n\times n}$ with $\rk[K,L]= n$ and $KL^*=LK^*$. Then, by \eqref{Lstar}, $\cM^*=\ran \begin{smallbmatrix}L^*\\ -K^*\end{smallbmatrix}$. Assume that $(x,y)\in\cM^*$. Then there exists some $z\in\K^n$ with $x=L^*z$ and $y=-K^*z$. This yields
\[[K,L]\begin{smallpmatrix}x\\y\end{smallpmatrix}=Kx+Ly=KL^*z-LK^*z=0.\] Altogether we obtain that $\cM^*\subset\cM$. On the other hand, we obtain from $\rk[K,L]= n$ that $\dim\cM=\dim\ker[K,L]=n$ and $\dim\cM^*=\rk\begin{smallbmatrix}L^*\\ -K^*\end{smallbmatrix}=n$, which, together with $\cM^*\subset\cM$ leads to $\cM^*=\cM$.\hfill
\end{proof}

\begin{remark}\label{rem:symequiv}
Note that Lemma~\ref{lem:symequiv} can be further modified to characterize skew-adjointness of a~linear relation $\cM$. In particular, it is analogous to prove that the following statements are equivalent.
\begin{itemize}
    \item[\rm (a)] $\cM$ is skew-adjoint,
    \item[\rm (b)] $\cM$ is skew-symmetric and $\dim\cM=n$,
    \item[\rm (c)] $\cM=\ker[K,L]$ for some $K,L\in\K^{n\times n}$ with $KL^*=-LK^*$ and $\rk[K,L]= n$.
\end{itemize}
Moreover, the following statements are equivalent.
\begin{itemize}
    \item[\rm (d)] $\cM$ is skew-symmetric,
    \item[\rm (e)] $\cM=\ran\begin{smallbmatrix}F\\ G\end{smallbmatrix}$ for some $F,G\in\K^{n\times l}$ with $G^*F=-F^*G$,
    \item[\rm (f)] $\re\langle x,y\rangle=0$ for all $(x,y)\in\cM$.
\end{itemize}
The equivalence of (d) and (e) can be derived from the same modifications, whereas the equivalence of (e) and (f) follows from considering
\[\re\langle x,y\rangle= \tfrac{1}{2}(\langle x,y\rangle+\langle y,x\rangle)=z^*(F^*G+G^*F)z,\] for $(x,y)=(Fz,Gz)\in\cM$ with $z\in\K^l$ given by the range representation $\cM=\ran\begin{smallbmatrix}
F\\G
\end{smallbmatrix}$.
\end{remark}


\begin{definition}[dissipative, nonnegative]\label{defdissip}
Let $\cM\subset\K^{2n}$ be a~linear relation. Then $\cM$ is called
\begin{itemize}
    \item[\rm (a)]\emph{dissipative}, if
\begin{align*}
\re\langle x,y\rangle\leq 0, \quad \text{for all $(x,y)\in\cM$}.
\end{align*}
\item[\rm (b)]\emph{nonnegative}, denoted by $\cM\geq 0$, if $\cM$ is symmetric with
\begin{align*}
\langle x,y\rangle\geq 0, \quad \text{for all $(x,y)\in\cM$}.
\end{align*}
\item[\rm (c)]\emph{maximally dissipative},
if it dissipative, and it is not a~proper subspace of a dissipative linear relation.
\item[\rm (d)]\emph{maximally nonnegative},
if it is nonnegative, and it is not a~proper subspace of a~nonnegative linear relation.
\end{itemize}
\end{definition}

\begin{figure}
	\centering
	\includegraphics[width=\textwidth]{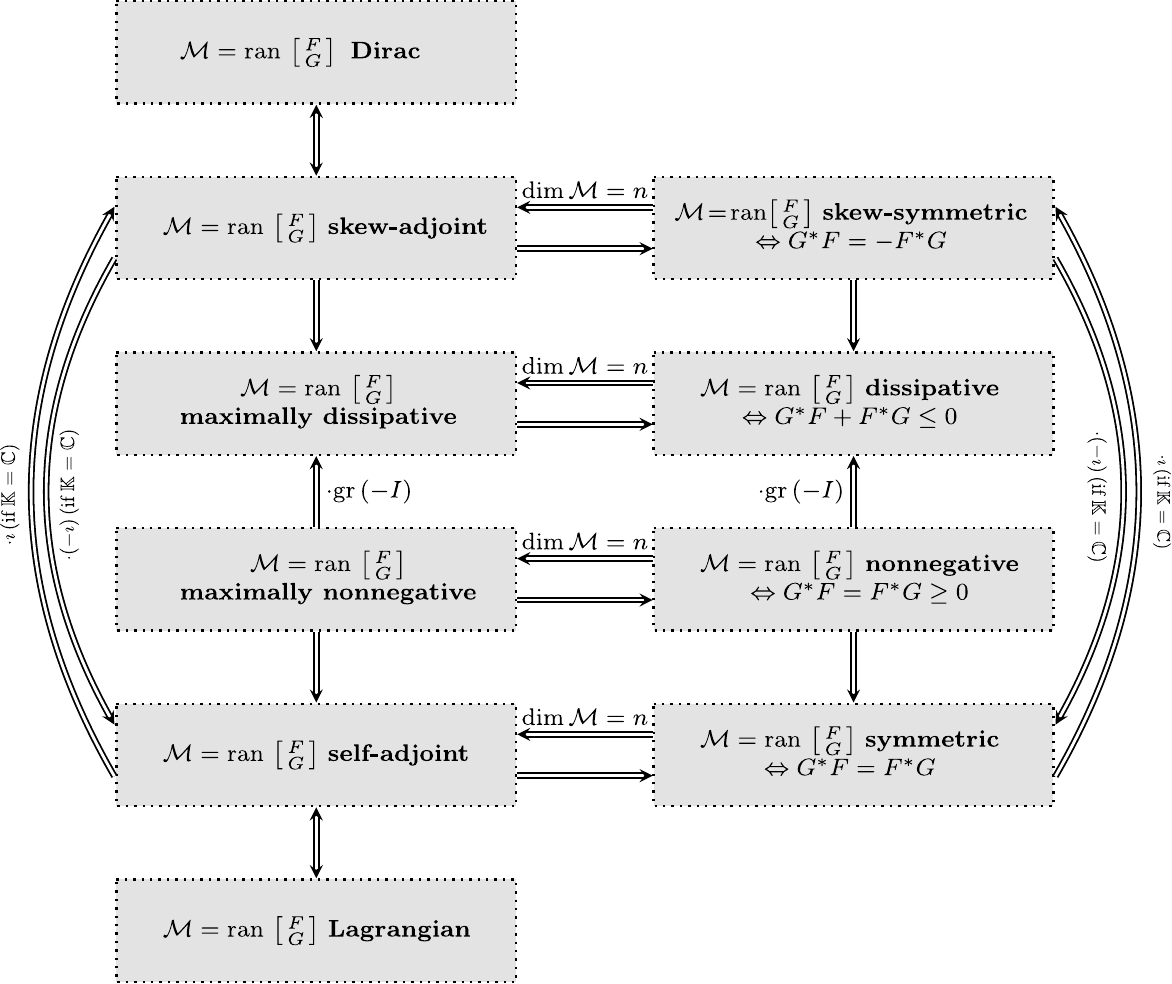}
	\caption{An overview of the structural assumptions on the subspace $\cM$ in range representation with $F,G\in\K^{n\times n}$.}
    \label{fig:subspaces}
\end{figure}

We would like to remark, that other definitions of dissipative linear relations exists in the literature. For example in \cite[Def.~1.6.1]{BehrHassdeSn20} a linear relation $\cM\subseteq\C^{2n}$ is called dissipative if $\im\langle x,y\rangle\geq 0$ for all $(x,y)\in\cM$. However, if $\cM$ is dissipative in the sense of Definition \ref{defdissip} then $-\imath\cM$ is dissipative in the aforementioned sense and vice versa. In the context of port-Hamiltonian systems, Dirac subspaces correspond exactly to the skew-adjoint linear relations, and Lagrange subspaces exaclty to the self adoint linear relations. In particular, Dirac subspaces are maximally dissipative linear relations, and Lagrangian subspaces are maximally nonnegative linear relations, but the converse is not true in general, see Figure~\ref{fig:subspaces}.

Now we collect some basic results on linear relations. As a consequence of Lemma \ref{lem:symequiv} and Remark \ref{rem:symequiv}, we can characterize nonnegativity and dissipativity as follows.
\begin{lemma}
\label{lem:nonneg}
Let $\cM=\ran\begin{smallbmatrix}F\\G\end{smallbmatrix}$ with $F,G\in\K^{n\times l}$ be a linear relation. Then $\cM$ is nonnegative if, and only if, $G^*F=F^*G\geq0$ and dissipative if, and only if, $G^*F+F^*G\leq 0$. Moreover, the following statements are equivalent.
\begin{itemize}
    \item[\rm (a)] $\cM$ is maximally nonnegative.
    \item[\rm (b)] $\cM$ is nonnegative and $\dim\cM=n$.
    \item[\rm (c)] $\cM$ is nonnegative and self-adjoint.
\end{itemize}
Further, $\cM$ is maximally dissipative if, and only if, $\dim\cM=n$ and $G^*F+F^*G\leq 0$.
\end{lemma}
\begin{proof}
For the first two equivalences, observe that the range representation yields
\[\langle x,y\rangle\geq 0, \quad \text{for all $(x,y)\in\cM$}\quad\Longleftrightarrow\quad z^*F^*Gz\geq 0,\quad\text{for all $z\in\K^n$}\]
and
\[\re\langle x,y\rangle\leq 0, \quad \text{for all $(x,y)\in\cM$}\quad\Longleftrightarrow\quad z^*(F^*G+G^*F)z\leq 0,\quad\text{for all $z\in\K^n$}.\]
The statements then follows directly from Lemma \ref{lem:symequiv}.
We now show the equivalences (a)-(c). \\
``(a)$\Longrightarrow$(b)'': Assume that $\cM$ is maximally nonnegative. Then it follows from the definition nonnegativity that $\cM^*$ is nonnegative as well. By the symmetry of $\cM$, we further have $\cM\subset\cM^*$, and maximality leads to $\cM=\cM^*$. Thus by Lemma~\ref{lem:symequiv}, $\dim\cM=n$.\\
``(b)$\Longrightarrow$(a)'': Let $\cM$ be nonnegative with $\dim\cM=n$. Then $\cM$ is in particular symmetric with $\dim\cM=n$, whence, by Lemma~\ref{lem:symequiv}, it is not a~proper subspace of a~symmetric relation. In particular, it is not a~proper subspace of a~nonnegative relation. That is, $\cM$ is maximally nonnegative.\\
``(b)$\Longleftrightarrow$(c)'': This equivalence is a direct consequence of the equivalence of the statements (a) and (b) of Lemma \ref{lem:symequiv}.\\
It remains to prove the last equivalence for dissipative relations. Assume that $\cM=\ran\begin{smallbmatrix}F\\G\end{smallbmatrix}$ is dissipative. First note
\[F^*G+G^*F=\begin{smallbmatrix}F\\G\end{smallbmatrix}^*\begin{smallbmatrix}0&I_n\\I_n&0\end{smallbmatrix}\begin{smallbmatrix}F\\G\end{smallbmatrix}\leq0\]
and that $\begin{smallbmatrix}0&I_n\\I_n&0\end{smallbmatrix}$ has $n$ positive and $n$ negative eigenvalues. If $\dim\cM>n$, then Sylvester's inertia theorem \cite[Thm.~4.5.8]{HornJohn13} yields that $F^*G+G^*F$ has to have at least one positive eigenvalue. Consequently, any $n$-dimensional dissipative relation is maximal. On the other hand, if $\cM$ is dissipative with $\dim\cM<n$, we can, again by employing Sylvester's inertia theorem, infer that $\cM$ can be further extended to a~linear relation which is still dissipative.
\hfill
\end{proof}

\begin{lemma}
\label{lem:dissprop}
Let $\cM=\ran\begin{smallbmatrix}F\\G\end{smallbmatrix}$ with $F,G\in\K^{n\times l}$ be a dissipative (symmetric) linear relation. Then $\dom\cM\subseteq(\mul\cM)^\perp$ and $\ran\cM\subseteq(\ker\cM)^{\perp}$.
Furthermore, the following three statements are equivalent:
\begin{enumerate}[(i)]
\item $\cM$ is maximally dissipative (self-adjoint).
\item $\cM$ is dissipative (symmetric) and $\dom\cM=(\mul\cM)^\perp$.
\item $\cM$ is dissipative (symmetric) and $\ran\cM=(\ker\cM)^\perp$.
\end{enumerate}
\end{lemma}
\begin{proof}
The statement $\dom\cM\subseteq(\mul\cM)^\perp$ as well as the implication ``(i)$\Longrightarrow$(ii)'' has been proven in \cite[Lem.~2.1]{AzizDijkWanj13} for the dissipative case, and in \cite[Prop.~1.3.2]{BehrHassdeSn20} for the symmetric case. Further, if $\cM$ is dissipative (symmetric), so is $\cM^{-1}$ by Lemma~\ref{lem:symequiv}. Hence, $\ker\cM=\mul(\cM^{-1})\subseteq\dom(\cM^{-1})^{\perp}=(\ran\cM)^{\perp}$.\\
``(ii)$\Longrightarrow$(i)'': Let $\cM$ be dissipative  or symmetric and, additionally, assume that $\dom\cM=(\mul\cM)^\perp$. For $k\coloneqq\dim\dom\cM$, let $(x_1,\ldots,x_k)$ be a~basis of $\dom\cM$. Then there exist $y_1,\ldots,y_k\in\K^n$, such that $(x_i,y_i)\in\cM$ for $i=1,\ldots,k$. Then we have
\[\Span\left\{(x_1,y_k),\ldots,(x_k,y_k)\right\}\cap (\{0\}\times\mul\cM)=\{0\}.\]
Since, further, $\{0\}\times\mul\cM\subseteq\cM$, we obtain that
\[\Span\left\{(x_1,y_k),\ldots,(x_k,y_k)\right\}\cap (\{0\}\times\mul\cM)\subset\cM,\]
and thus
\[\dim\cM\geq \dim\dom\cM+\dim\mul\cM=\dim(\mul\cM)^\perp+\dim\mul\cM=n.\]
Then Lemma~\ref{lem:nonneg} (resp.\ Lemma~\ref{lem:symequiv}) imply that $\cM$ is maximally dissipative (self-adjoint).\\
``(ii)$\Longleftrightarrow$(iii)'': This follows by the already proven equivalence between (i) and (ii), together with $\dom\cM=\ran\cM^{-1}$, $\mul\cM=\ker\cM^{-1}$, and the fact that $\cM$ is dissipative (maximally dissipative, symmetric, self-adjoint) if, and only if, the inverse $\cM^{-1}$ has the respective property.
\hfill
\end{proof}

\begin{proposition}
\label{prop:graph}
Let $\cM=\ran\begin{smallbmatrix}
F\\G
\end{smallbmatrix}$ with  $F,G\in\K^{n\times l}$ be a linear relation with $\dim\cM=n$. Then $\cM=\gr M$ for some $M\in\K^{n\times n}$ if, and only if, $\rk F=n$.\\
In this case, $\cM$ is self-adjoint (skew-adjoint, maximally nonnegative, maximally dissipative) if, and only if, $M$ is Hermitian (skew-Hermitian, positive semi-definite, dissipative).
\end{proposition}
\begin{proof}
Let $\cM=\ran\begin{smallbmatrix}F\\ G\end{smallbmatrix}$ with $\dim\cM=n$.
If $\cM=\gr M$ for some $M\in\K^{n\times n}$ then $\ran F=\dom\cM=\K^n$ which implies $\rk F=n$. Conversely, let $F\in\K^{n\times l}$ be given with $\rk F=n$. Then $\dom\cM=\ran F=\K^n$. Consider the canonical basis $(e_1,\ldots,e_n)$ of $\K^n$. Then there exist $x_1,\ldots,x_n$ with $Fx_i=e_i$ for $i=1,\ldots,n$. Define
\[M\coloneqq[Gx_1,\ldots,Gx_n]\in \K^{n\times n}.\]
Then, by
$\begin{smallbmatrix}
F\\G
\end{smallbmatrix}x_i=\begin{smallpmatrix}
Fx_i\\Gx_i
\end{smallpmatrix}=\begin{smallpmatrix}
e_i\\Me_i
\end{smallpmatrix}=\begin{smallbmatrix}
I_n\\M
\end{smallbmatrix}e_i$,
we obtain
\[\ran\begin{smallbmatrix}
I_n\\M
\end{smallbmatrix}\subset\ran \begin{smallbmatrix}
F\\G
\end{smallbmatrix}\]
However, since the dimensions of both spaces equal, we even have equality.\\
The second part of the result follows from Lemma~\ref{lem:dissprop} and Lemma~\ref{lem:nonneg}.\hfill
\end{proof}

We close this section with a~technical result, where we present a~certain range representation of the product of a~dissipative and a~symmetric subspace. A~proof of the following proposition can be found in the appendix.
\begin{proposition}
\label{lem:product}
Let $\Dc\subseteq \K^{2n}$ be a~dissipative and $\Lc\subseteq \K^{2n}$ be a~symmetric linear relation, and assume that $\ker\Lc\cap\mul\Dc=\{0\}$. Let $n_1=\dim(\ran\Lc\cap\dom\Dc)$ and $n_2=n-n_1$. Then there exists some unitary matrix $U\in\K^{n\times n}$, such that the product of  $\Dc$ and $\Lc$ has a~representation
\begin{align}
\label{dl_final}
\Dc\Lc=\ran\diag(U,U)\begin{smallbmatrix}
L_{11}&0\\L_{21}& L_{22}\\D_{11}& 0\\ D_{21}& D_{22}
\end{smallbmatrix}
\end{align}
for some matrices $L_{ij},D_{ij}\in\K^{n_i\times n_j}$ with
\begin{align}
\label{L11D11_prop}
L_{11}=&\,L_{11}^*,\quad   &&D_{11}+D_{11}^*\leq 0,\\
L_{22}=&\,L_{22}^2=L_{22}^*, && -D_{22}= D_{22}^2=-D_{22}^*,\quad &&\ran L_{22}\cap\ran D_{22}=\{0\}.\label{l22d22_prop}
\end{align}
Moreover, the following holds:
\begin{itemize}
\item[\rm (i)]
If $\Lc$ is nonnegative then $L_{11}$ is positive semi-definite. If, additionally, $\Lc$ is maximal then $\ker L_{11}\subset\ker L_{21}$.
\item[\rm (ii)]
If $\Dc$ is skew-symmetric then $D_{11}$ is skew-Hermitian. 
\item[\rm (iii)] $\ker L_{22}\cap\ker D_{22}=\{0\}$ if, and only if,  
\[\mul\Dc\widehat \dotplus\ker\Lc=(\ran\Lc)^\perp\widehat +(\dom\Dc)^\perp.\]
\item[\rm (iv)] If, additionally, $\Dc=\gr D$ for some dissipative $D\in\K^{n\times n}$ and $\Lc$ is self-adjoint, then $L_{21}=D_{22}=0$ and $L_{22}=I_{n_2}$. Furthermore, we have
\[\begin{aligned}
\ker L_{11}\times \{0\}=&\,U^*\mul\Lc,\\
\ker D_{11}\times\{0\}=&\,U^*\left\{x\in\ran\Lc\,\vert\;Dx\in \ker\Lc\right\}.
\end{aligned}\]
\item[\rm (v)] If, additionally, $\Dc$ is maximally dissipative and $\Lc=(\gr L)^{-1}$ for some $L\in\K^{n\times n}$, then $L$ is Hermitian, and $D_{22}=-I_{n_2}$, $D_{21}=L_{22}=0$. Furthermore, we have
\[\begin{aligned}
\ker L_{11}\times \{0\}=&\,U^*\left\{x\in\dom\Dc\,\vert\;Lx\in \mul\Dc\right\},\\
\ker D_{11}\times\{0\}=&\,U^*\ker \Dc.
\end{aligned}\]
\end{itemize}
\end{proposition}

\section{Port-Hamiltonian formulation via linear relations}
\label{sec:comp}
Our ongoing focus will be placed on image representations \eqref{factor3} 
for a~dissipative linear relation $\Dc\subset\K^{2n}$ and a~symmetric linear relation $\Lc\subset\K^{2n}$, and we will investigate the properties of the pencil $sE-A$.

Before we start with such an investigation, we will briefly highlight the connection between the DAE $\tfrac{{\rm d}}{{\rm d}t}Ez(t)=Az(t)$ and
differential inclusion \eqref{MSdefintro} in the case where the range representation \eqref{factor3} holds. To this end, assume that $\Dc,\Lc\subset\K^{2n}$ are linear relations and $E,A\in\K^{n\times m}$, such that \eqref{factor3} holds.\\
Assuming that the $\K^m$-valued function $z(\cdot)$ solves the DAE $\tfrac{{\rm d}}{{\rm d}t}Ez(t)=Az(t)$ on an interval $I\subset\R$, we obtain that $x(\cdot)\coloneqq Ez(\cdot)$ fulfills
\[\forall \,t\in I:\quad\begin{smallpmatrix}
    x(t)\\\dot{x}(t)
  \end{smallpmatrix}=\begin{smallpmatrix}
    Ez(t)\\\tfrac{{\rm d}}{{\rm d}t}Ez(t)
  \end{smallpmatrix}=\begin{smallpmatrix}
    Ez(t)\\Az(t)
  \end{smallpmatrix}=\begin{smallbmatrix}
    E\\A
  \end{smallbmatrix}z(t)\in \ran\begin{smallbmatrix}E\\ A\end{smallbmatrix}=\Dc\Lc.\]
  By definition of the product of linear relations, this leads to the existence of some $e(\cdot):I\to\K^n$ such that \eqref{MSdefintro} holds for all $t\in I$.\\
On the other hand, if $x(\cdot),e(\cdot):I\to\K^n$ fulfill \eqref{MSdefintro}, then we obtain, again by the definition of the product of linear relations, that $(x(t),\dot{x}(t))\in\Dc\Lc$, and thus
\[\forall \,t\in I:\quad\begin{smallpmatrix}
    x(t)\\\dot{x}(t)
  \end{smallpmatrix}\in\Dc\Lc=\ran\begin{smallbmatrix}E\\ A\end{smallbmatrix}.\]
This leads to the existence of some $z(\cdot):I\to\K^m$ with
\[\begin{smallpmatrix}
    x(t)\\\dot{x}(t)
  \end{smallpmatrix}=\begin{smallbmatrix}E\\ A\end{smallbmatrix}z(t),\]
and thus
\[\forall\,t\in I:\quad\tfrac{{\rm d}}{{\rm d}t}Ez(t)=\dot{x}(t)=Az(t).\]
In \cite{MascvdSc18}, $\Dc,\Lc$ were assumed to be a~Dirac and a~Lagrangian subspace, respectively. In the language of linear relations, this means that $\Dc$ is skew-adjoint and $\Lc$ is self-adjoint. As mentioned before, we consider a~slightly larger class. Namely, instead of skew-adjoint and self-adjoint linear relations, we allow for dissipative $\Dc$, whereas $\Lc$ is allowed to be only symmetric. This is a~generalization in two respects: First of all, the relations $\Dc$ and $\Lc$ may have a~dimension less than $n$ and, second, we allow for relations $\Dc$ with $\re\langle x,y\rangle\leq0$ instead of $\re\langle x,y\rangle=0$ for all $(x,y)\in\Dc$.

Note that, in the special case where both $\Dc$ and $\Lc$ are graphs, i.e., $\Dc=\gr D$, $\Lc=\gr Q$ for some $D,Q\in\K^{n\times n}$, then the dissipativity of $\Dc$ leads to the dissipativity of $D$, and the symmetry of $\Lc$ means that $Q$ is Hermitian, and we end up with $z(t)=x(t)$ and an ordinary differential equation $\dot{x}(t)=DQx(t)$, which is port-Hamiltonian in the classical sense, see \cite{vdScJelt14}.


Our motivation for considering the above class involving dissipative and symmetric relation is that it also comprises the one treated in \cite{MehlMehrWojt18}. To this end, recall that a~DAE
$\tfrac{{\rm d}}{{\rm d}t}Ez(t)=Az(t)$ with $E,A\in\K^{n\times m}$ has in \cite{MehlMehrWojt18} been defined to be port-Hamiltonian, if there exist $D\in\K^{n\times n}$, $Q\in\K^{n\times m}$ with $A=DQ$, $D+D^*\leq 0$ and $Q^*E=E^*Q$.
It can be seen that, by the definition of the product of linear relations, for $\Dc=\gr D$ and $\Lc=\ran\begin{smallbmatrix}E\\Q\end{smallbmatrix}$, it holds
\begin{equation}\label{eq:DL-^1rel}
\begin{aligned}
\Dc\Lc=&\,\{ (x_1,x_2)\in\K^{2n} \,|\, \text{$\exists y\in\K^n$ s.t.\ } (x_1,y)\in\Lc\,\wedge\, (y,z_2)\in\Dc \}\\
=&\,\{(x_1,x_2)\in\K^{2n}\,|\,\text{$\exists z,y\in\K^n$ s.t.\ } (x_1,y)=(Ez,Qz)\in\Lc\,\wedge\, x_2=Dy \}\\
=&\,\{(x_1,x_2)\in\K^{2n}\,|\,\text{$\exists z\in\K^n$ s.t.\ }x_1=Ez\,\wedge\, x_2=DQz \}\\
=&\,\ran\begin{smallbmatrix}E\\DQ\end{smallbmatrix}.
\end{aligned}
\end{equation}
In particular, it holds \eqref{factor3} for $A=DQ$, whence the function $x(\cdot)\coloneqq Ez(\cdot)$ indeed fulfills $(x,\dot{x})\in\Dc\Lc$. The dissipativity of $D\in\K^{n\times n}$ leads, via Lemma~\ref{lem:nonneg}, to the maximal dissipativity of $\Dc$, whereas, by Lemma~\ref{lem:symequiv}, $\Lc$ is symmetric (but not necessarily self-adjoint).

Summarizing from the previous findings, the differences between the approaches to pH-DAEs in
\cite{MehlMehrWojt18} and \cite{MascvdSc18} are the following:\\

\begin{table}[H]
\begin{center}\fbox{
  \begin{minipage}[t]{10cm}
~\\[-4mm]
\begin{enumerate}[(i)]
    \item $\ran\begin{smallbmatrix}
Q\\E
\end{smallbmatrix}$ needs to be $n$-dimensional in \cite{MascvdSc18}, whereas, in \cite{MehlMehrWojt18}, it might have a~smaller dimension.\\[-2mm]
\phantom{x}
\item the relation $\Dc$ needs to be a~graph of a~matrix in \cite{MehlMehrWojt18}, whereas, in \cite{MascvdSc18}, $\Dc$ might have a~multi-valued part.\\[-2mm]
\phantom{x}
\item the relation $\Dc$ is skew-adjoint in \cite{MascvdSc18}, whereas, in \cite{MehlMehrWojt18}, $\Dc$~might be dissipative.\\[-2mm]
\phantom{x}
\end{enumerate}
\end{minipage}
}\caption{Differences between the approaches in \cite{MascvdSc18} and \cite{MehlMehrWojt18}}\label{fig:differences}
\end{center}\end{table}
\noindent
This justifies to prescribe the following terminology.
\begin{definition}[Port-Hamiltonian matrix pencil]\label{def:pHpenc}\
We call a~matrix pencil $s E-A\in\K[s]^{n\times m}$
\begin{enumerate}[(i)]
\item {\em port-Hamiltonian (pH) in the sense of \cite{MehlMehrWojt18}}, if there exist  $E,Q\in\K^{n\times m}$ and $D\in\K^{n\times n}$ such that $D$ is dissipative, $A=DQ$ and $E^*Q=Q^*E$,
\item {\em port-Hamiltonian in the sense of \cite{MascvdSc18}}, if \eqref{factor3} holds for some~skew-adjoint linear relation $\Dc\subset\K^{2n}$ and some~self-adjoint linear relation $\Lc\subset\K^{2n}$, and
\item {\em port-Hamiltonian in our sense}, if \eqref{factor3} holds for some~dissipative linear relation $\Dc\subset\K^{2n}$ and some~symmetric linear relation $\Lc\subset\K^{2n}$.
\end{enumerate}
\end{definition}
It can be directly seen that pencils which are pH in the sense of \cite{MascvdSc18} or pH in the sense of \cite{MehlMehrWojt18} are also pH in our sense. The reverse statements are not true as the following examples show. Thereafter, we present conditions on a~ pencil which is pH in the sense of \cite{MascvdSc18} to be also pH in the sense of \cite{MehlMehrWojt18}, and vice-versa.

We start with presenting a~system in which (i) in Fig.~\ref{fig:differences} is the reason why it is pH in the sense of \cite{MehlMehrWojt18}, but not in the sense of \cite{MascvdSc18}.
\begin{example}\label{ex:ex1}
Let $E=Q=\begin{smallbmatrix}1\\0\end{smallbmatrix}$, $A=\begin{smallbmatrix}
0\\1
\end{smallbmatrix}$ and  $D=\begin{smallbmatrix}
0&-1\\1&0
\end{smallbmatrix}$. Then $A=DQ$ and  $Q^*E=1=E^*Q$, i.e.\ $sE-A$ is pH in the sense of \cite{MehlMehrWojt18}.\\
Next we show that it is not pH in the sense of \cite{MascvdSc18}. Seeking for a contradiction, assume that $\Dc,\Lc\subseteq\C^4$ be skew-adjoint and self-adjoint subspaces such that
\begin{align}
\label{sosiehtsaus}
\ran\begin{bmatrix}
E\\A
\end{bmatrix}=\Span\left\{\begin{smallpmatrix}
1\\0\\0\\1
\end{smallpmatrix}\right\}=\Dc\Lc.
\end{align}
Then we see that $\mul\Dc\Lc=\ker\Dc\Lc=\{0\}$, which gives $\mul\Dc=\ker\Lc=\{0\}$. This together with Lemma~\ref{lem:dissprop} yields, by invoking $\ran\Lc=\dom \Lc^{-1}$, that $\dom\Dc=\ran\Lc=\K^2$, and we infer, from Proposition~\ref{prop:graph} that $\Dc=\gr \hat D$ and $\Lc=(\gr E)^{-1}$ for some skew-Hermitian $\hat D\in\K^{2\times 2}$ and some Hermitian $E\in\K^{2\times 2}$.
Hence we can rewrite \eqref{sosiehtsaus} as
\begin{align}
\label{sosiehtsbesseraus}
\Span\left\{\begin{smallpmatrix}
1\\0\\0\\1
\end{smallpmatrix}\right\}=\ran\begin{bmatrix}E\\ \hat D\end{bmatrix}.
\end{align}
Denoting the $i$th canonical unit vector by $e_i$, this gives
 \begin{align*}
\ran E&=\Span\left\{e_1\right\},& \ran(\hat D^*)&=\ran \hat D=\Span\left\{e_2\right\}.
\end{align*}
Since the space on the left hand side in \eqref{sosiehtsbesseraus} is one-dimensional, we obtain $\ker E\cap\ker \hat D\neq\{0\}$.
On the other hand \eqref{sosiehtsbesseraus}, $E=E^*$ and $\hat D=-\hat D^*$ leads to
\begin{align*}
\ker E&=(\ran E^*)^\perp=\Span\left\{e_2\right\},& \ker \hat D&=(\ran \hat D^*)^\perp=\Span\left\{e_1\right\}.
\end{align*}
This implies $\ker E\cap\ker \hat D=\{0\}$, which is a contradiction to the already proven fact that $\ker E\cap\ker \hat D$ is a~non-trivial space. Consequently, the pencil $sE-A$ cannot be pH in the sense of \cite{MascvdSc18}.
\end{example}

\noindent
Our second example is one which is pH-DAE in the sense of \cite{MascvdSc18} but not in the sense of \cite{MehlMehrWojt18}. The reason for the latter will be in Fig.~\ref{fig:differences}, i.e., it does not admit a~representation \eqref{factor3}
in which $\Dc$ is a~graph.


\begin{example}\label{ex:notmmw}
Consider
\begin{align*}\Dc=\,\ran\begin{smallbmatrix}1&0&0\\0&1&0\\0&0&0\\0&1&0\\-1&0&0\\0&0&1\end{smallbmatrix}\subseteq \K^6,\quad
\Lc=\,\ran\begin{smallbmatrix}1&0&0\\0&1&0\\0&0&1\\1&0&0\\0&1&0\\0&0&0\end{smallbmatrix}\subseteq \K^6.
\end{align*} Then, by using Lemma\,\ref{lem:symequiv} and Remark~\ref{rem:symequiv}, it can be seen that $\Dc$ skew-adjoint and $\Lc$ is self-adjoint. It can be seen that both $\mul\Dc$ and $\ker\Lc$ are spanned by the third canonical unit vector, and
\[\Dc\Lc=\begin{smallbmatrix}1&0&0&0\\0&1&0&0\\0&0&1&0\\0&1&0&0\\-1&0&0&0\\0&0&0&1\end{smallbmatrix}.
\]
Assume that $\Dc\Lc=(\gr \hat D)\hat\Lc$ with $\hat D\in\K^{3\times 3}$ and symmetric $\hat \Lc\subset\K^6$. The symmetry of $\hat \Lc$ yields
\[
4=\dim\Dc\Lc=\dim(\gr \hat D)\hat\Lc\leq\dim\hat \Lc \leq 3,
\]
which is a contradiction. Hence, rewriting $\Dc\Lc=(\gr \hat D)\hat\Lc$ is not possible, whence $sE-A$ is not pH in the sense of \cite{MehlMehrWojt18}.
\end{example}

Our last is example is one
which is pH in the sense of \cite{MehlMehrWojt18}, but not in the sense of \cite{MascvdSc18}. To disprove that this system is pH the sense of \cite{MascvdSc18}, we show that there is no representation
\eqref{factor3}
with skew-symmetric $\Dc$ and symmetric $\Lc$, cf.\ (iii) in Fig.~\ref{fig:differences}.
\begin{example}
\label{ex:notms}
Let $E=Q=-D=-A=1\in\R^{1\times 1}$. Then, clearly, $A=DQ$ and  $Q^*E=1=E^*Q$, i.e.,\ $sE-A$ is pH in the sense of \cite{MehlMehrWojt18}. Then
\begin{equation}
\ran\begin{smallbmatrix}
E\\A
\end{smallbmatrix}=\Span\left\{\begin{smallpmatrix}
1\\-1
\end{smallpmatrix}\right\}\label{eq:spanea11}.
\end{equation}
Now assume that \eqref{factor3} holds for some skew-symmetric linear relation $\Dc\subset\R^2$ and symmetric $\Lc\subset\R^2$. As $\Dc\subset\R^2$ is skew-symmetric, we immediately obtain that it is either trivial, or it is spanned by the first or second canonical unit vector in $\R^2$. In the first two cases $\Dc=\{0\}$ and $\Dc=\Span\left\{\begin{smallpmatrix}
1\\0
\end{smallpmatrix}\right\}$, we have $y=0$ for all $(x,y)\in \Dc\Lc$, which contradicts to \eqref{eq:spanea11}. On the other hand, if $\Dc=\Span\left\{\begin{smallpmatrix}
0\\1
\end{smallpmatrix}\right\}$, we have $\begin{smallpmatrix}
0\\1
\end{smallpmatrix}\in \Dc\Lc$, which is again a~contradiction to \eqref{eq:spanea11}.
\end{example}

After having highlighted the differences between the approaches of \cite{MascvdSc18} and \cite{MehlMehrWojt18}, we now analyze their mutualities. That is, we give conditions on a~matrix pencil which is pH in the sense of \cite{MascvdSc18} to be pH in the sense of \cite{MehlMehrWojt18}, and vice-versa.
\begin{proposition}
Assume that $sE-A\in\K[s]^{n\times m}$ is pH in the sense of \cite{MehlMehrWojt18}, i.e., $A=DQ$ for some dissipative $D\in\K^{n\times n}$ and $Q\in\K^{n\times m}$.\\
If, additionally $D+D^*=0$ and
$\dim\ran\begin{smallbmatrix}
E\\Q
\end{smallbmatrix}=n$,
then $sE-A$ is pH in the sense of \cite{MascvdSc18} with, in particular, \eqref{factor3} holds for
$\Lc\coloneqq\ran\begin{smallbmatrix}
E\\Q
\end{smallbmatrix}$ and $\Dc=\gr D$.
\end{proposition}
\begin{proof}
Assume that $E,A,Q\in\K^{n\times m}$ fulfill $A=DQ$, $D+D^*=0$, $E^*Q=Q^*E$ and $\dim\ran\begin{smallbmatrix}
E\\Q
\end{smallbmatrix}=n$. Then, by $\re \langle x,Dx\rangle=0$ for all $x\in\K^n$, we have that $\Dc\coloneqq \gr D$ is skew-symmetric. Since, further, $\dim\gr D=n$, Lemma~\ref{lem:nonneg} implies that $\Dc$ is even skew-adjoint. Moreover, by using Lemma~\ref{lem:symequiv}, $\dim\ran\begin{smallbmatrix}
E\\Q
\end{smallbmatrix}=n$ and $E^*Q=Q^*E$ imply that $\Lc\coloneqq \ran\begin{smallbmatrix}
E\\Q
\end{smallbmatrix}$ is self-adjoint. Then the result follows since, by \eqref{eq:DL-^1rel}, \eqref{factor3} holds for $A=DQ$.\hfill
\end{proof}

\begin{figure}
{\footnotesize
\begin{center}
  \resizebox{0.98\linewidth}{!}{\begin{tikzpicture}[thick,node distance = 22ex,implies/.style={double,double equal sign distance,-implies}, box/.style={fill=white,rectangle, draw=black},
  blackdot/.style={inner sep = 0, minimum size=3pt,shape=circle,fill,draw=black},plus/.style={fill=white,circle,inner sep = 0,very thick,draw},
  metabox/.style={inner sep = 3ex,rectangle,draw,dotted,fill=gray!20!white}]

    \node (vdS)     [metabox,minimum size=2.5cm]  {\parbox{2.8cm}{\rm\centering $sE-A$ is pH in the \\sense of \cite{MascvdSc18}}};
    \node (MM)     [right of =vdS,metabox,xshift=41ex,minimum size=2.5cm]  {\parbox{2.8cm}{\rm\centering $sE-A$ is pH in the \\sense of \cite{MehlMehrWojt18}}};
    \node (our)     [right of =vdS,metabox,minimum size=1.5cm,yshift=-3cm,xshift=1.15cm]  {\parbox{2.8cm}{\rm\centering $sE-A$ is pH \\in our sense}};

\node (vdS1) [right of =vdS, xshift=27.5ex,yshift=2ex]{};
\node (vdS2) [right of =vdS, xshift=27.5ex,yshift=7ex]{};
\node (vdS3) [right of =vdS, xshift=27.5ex,yshift=-3ex]{};
\node (vdS4) [right of =vdS, xshift=27.5ex,yshift=-8ex]{};
\node (MM1) [left of =MM, xshift=-27.5ex,yshift=2ex]{};
\node (MM2) [left of =MM, xshift=-27.5ex,yshift=7ex]{};
\node (MM3) [left of =MM, xshift=-27.5ex,yshift=-3ex]{};
\node (MM4) [left of =MM, xshift=-27.5ex,yshift=-8ex]{};

    \draw[->,semithick,double,double equal sign distance,>=stealth,negated] (vdS1)  -- node [left,yshift=2ex,xshift=10ex] {Example~\ref{ex:ex1}\&\ref{ex:notms}}  (MM1);
    \draw[<-,semithick,double,double equal sign distance,>=stealth,negated] (vdS2)  -- node [left,yshift=2ex,xshift=8ex] {Example~\ref{ex:notmmw}}  (MM2);
    \draw[<-,semithick,double,double equal sign distance,>=stealth] (vdS3)  -- node [left,yshift=2ex,xshift=9ex] {$\mul\Dc=\{0\}$}  (MM3);
    \draw[->,semithick,double,double equal sign distance,>=stealth] (vdS4)  -- node [left,yshift=2ex,xshift=13ex] {$\dim\Lc=n$, $D\!+\!D^*=0$}  (MM4);
    \draw[->,semithick,double,double equal sign distance,>=stealth] (our)  -- node [left,yshift=2ex,xshift=9ex] {}  (MM);
    \draw[->,semithick,double,double equal sign distance,>=stealth] (our)  -- node [left,yshift=2ex,xshift=9ex] {}  (vdS);

  \end{tikzpicture}}

\end{center}}

  \caption{Relations between the port-Hamiltonian concepts from Definition~\ref{def:pHpenc}, with matrices $E,A\in\K^{n\times m}$, $D\in\K^{n\times n}$ and subspaces $\Dc,\Lc\subset\K^{2n}$.}
  \end{figure}
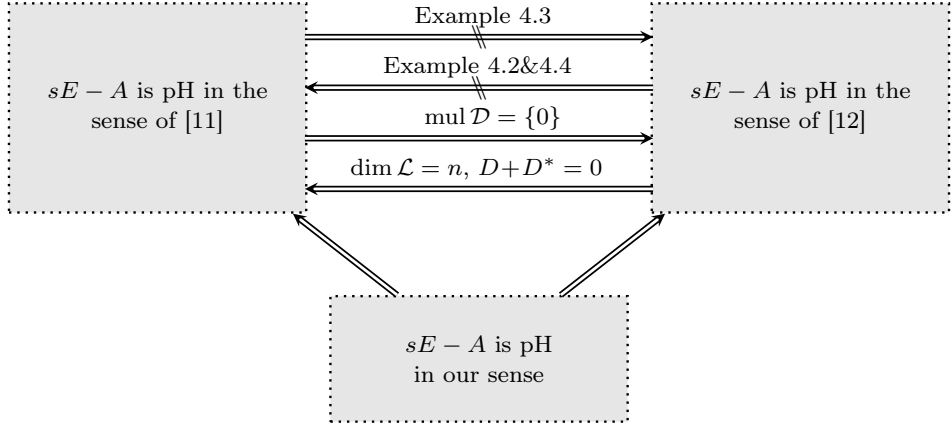

\begin{proposition}
Assume that $sE-A\in\K[s]^{n\times m}$ is pH in the sense of \cite{MascvdSc18}, i.e., \eqref{factor3} holds for some skew-adjoint $\Dc\subset\K^{2n}$ and some self-adjoint $\Lc\subset\K^{2n}$.\\
If, additionally $\mul\Dc=\{0\}$, then $sE-A$ is pH in the sense of \cite{MehlMehrWojt18}.\\
Namely, there exists some $Q\in\K^{n\times m}$ and some skew-Hermitian $D\in\K^{n\times n}$, such that $A=DQ$ and $E^*Q=Q^*E$. These
matrices fulfill $\Dc=\gr D$ and $\Lc\supseteq\ran\begin{smallbmatrix}
E\\Q
\end{smallbmatrix}$.
\end{proposition}
\begin{proof}
Assume that $sE-A\in\K[s]^{n\times m}$ fulfills \eqref{factor3} for some skew-adjoint $\Dc\subset\K^{2n}$ with $\mul\Dc=\{0\}$, and some $\Lc\subset\K^{2n}$. Then, by Remark~\ref{rem:symequiv}, $\dim\Dc=n$, whence there exist $F,G\in\K^{n\times n}$, such that $\Dc=\ran\begin{smallbmatrix}F\\G\end{smallbmatrix}$. The property $\mul\Dc=\{0\}$ further leads to $\ker F=\{0\}$, whence, by Proposition~\ref{prop:graph}, $\Dc=\gr D$ for some skew-Hermitian $D\in\K^{n\times n}$. Further, the self-adjointness of $\Lc$ leads, by using Lemma~\ref{lem:symequiv}, to the existence of some $E_1,Q_1\in\K^{n\times n}$ with $E_1^*Q_1=Q_1^*E_1$ and $\Lc=\ran\begin{smallbmatrix}E_1\\Q_1\end{smallbmatrix}$. The latter matrix has moreover full column rank since self-adjointness of $\Lc$ implies, by Lemma~\ref{lem:symequiv}, that $\dim\Lc=n$. Now, by making use of \eqref{eq:DL-^1rel}, we obtain
\[\ran\begin{smallbmatrix}E\\A\end{smallbmatrix}=\Dc\Lc=\ran\begin{smallbmatrix}E_1\\DQ_1\end{smallbmatrix}.\]
Consequently, there exists some $T\in\K^{n\times m}$ with
\[\begin{smallbmatrix}E\\A\end{smallbmatrix}=\begin{smallbmatrix}E_1\\DQ_1\end{smallbmatrix}T=\begin{smallbmatrix}E_1T\\DQ_1T\end{smallbmatrix},\]
which implies that $A=DQ$ for $Q=Q_1T$, and
\[\Lc=\ran \begin{smallbmatrix}E_1\\Q_1\end{smallbmatrix}\supseteq \ran \begin{smallbmatrix}E_1\\Q_1\end{smallbmatrix}T=\ran \begin{smallbmatrix}E\\Q\end{smallbmatrix}.\]
Invoking $E=E_1T$, we obtain that \[E^*Q=T^*E_1^*Q_1T=T^*Q_1^*E_1T=Q^*E\]
and the desired statement follows.\hfill
\end{proof}

\section{Regularity of port-Hamiltonian pencils}
\label{sec:reg}

In this section, we study regularity of square pencils $sE-A\in\K[s]^{n\times n}$ which are port-Hamiltonian in our sense, i.e.,
$E,A\in\K^{n\times n}$ fulfill \eqref{factor3} for a~dissipative relation $\Dc\subset\K^{2n}$ and a~symmetric relation  $\Lc\subset\K^{2n}$.
We start with a~characterization of regularity under the additional assumption that the multi-valued part of $\Dc$ and the kernel of $\Lc$ intersect trivially. 
\begin{proposition}
\label{cor:reg}
Let $sE-A\in\K^{n\times n}$ be pH in our sense, that is, \eqref{factor3} holds for some dissipative relation $\Dc\subset\K^{2n}$ and some symmetric relation $\Lc\subset\K^{2n}$. If $\mul\Dc\cap\ker\Lc=\{0\}$, then there exists a~unitary matrix $U\in\K^{n\times n}$ and an invertible matrix $T\in\K^{n\times n}$, such that, for some $n_1,n_2\in\N$,
\begin{align}
\label{meinblock}
U^*(sE-A)T=\begin{bmatrix}sL_{11}-D_{11}&0\\sL_{21}-D_{21}&sL_{22}-D_{22}\end{bmatrix}
\end{align}
with $L_{ij},D_{ij}\in\K^{n_i\times n_j}$, $i,j=1,2$, satisfying $L_{11}=L_{11}^*$, $D_{11}+D_{11}^*\leq 0$, $L_{22}=L_{22}^2=L_{22}^*$ and $-D_{22}=D_{22}^2=-D_{22}^*$.\\
Moreover, $sE-A$ is regular if, and only if, the following two conditions hold.
\begin{itemize}
\item[\rm (i)] $sL_{11}-D_{11}$ is regular, and
\item[\rm (ii)]  $\ker\Lc\widehat+\mul\Dc=(\ran\Lc)^{\perp}\widehat +(\dom \Dc)^{\perp}$. 
\end{itemize}
\end{proposition}
\begin{proof}
By Proposition~\ref{lem:product}, there exists a~unitary matrix $U\in\K^{n\times n}$, such that
\[
\ran\begin{bmatrix}E\\A\end{bmatrix}=\Dc\Lc=\ran\diag(U,U)\begin{smallbmatrix}
L_{11}&0\\L_{21}&L_{22}\\D_{11}&0\\D_{21}&D_{22}
\end{smallbmatrix}
\]
with $L_{ij},D_{ij}\in\K^{n_i\times n_j}$ having the desired properties. 
Hence there exists some invertible $T\in\K^{n\times n}$, such that
\[
\begin{bmatrix}E\\A\end{bmatrix}T=\diag(U,U)\begin{smallbmatrix}
L_{11}&0\\L_{21}&L_{22}\\D_{11}&0\\D_{21}&D_{22}
\end{smallbmatrix},
\]
which shows \eqref{meinblock}. For the proof of the remaining statement, we make use of the identity
\begin{equation}
\det(sE-A)=\det(T)^{-1}\det(U)\det(s L_{11}- D_{11})\det(sL_{22}- D_{22}).\label{eq:triangpenc}
\end{equation}
We first show that the regularity of $sE-A$ implies (i) and (ii): Assuming that $sE-A$ is regular, we obtain from \eqref{eq:triangpenc} that
both pencils $s L_{11}- D_{11}$ and $sL_{22}- D_{22}$ are regular. In particular, (i) holds, and $\ker L_{22}\cap\ker D_{22}=\{0\}$.
By Proposition~\ref{lem:product}~(iii), the latter implies the identity in (ii).\\
To prove the reverse implication, assume that the pencil $s L_{11}- D_{11}$ is regular and (ii) holds. Invoking,  Proposition~\ref{lem:product}~(iii), the condition (ii) implies $\ker L_{22}\cap\ker D_{22}=\{0\}$. Using  $L_{22}=L_{22}^2=L_{22}^*$ and $-D_{22}=D_{22}^2=-D_{22}^*$, the pencil $sL_{22}-D_{22}$ is positive real with $\ker L_{22}\cap\ker D_{22}=\{0\}$. Therefore, by Lemma~\ref{lem:posreal}, the pencil $sL_{22}-D_{22}$ is regular.
Then \eqref{eq:triangpenc} yields that $sE-A$ is regular.\hfill
\end{proof}


We apply Proposition~\ref{cor:reg} to the special case that $\Dc=\gr D$ from some dissipative $D\in\K^{n\times n}$.
\begin{corollary}\label{prop:EAreg}
Let $E,D,Q\in\K^{n\times n}$ with $Q^*E=E^*Q$ and $D+D^*\leq 0$. Consider the following three statements.
\begin{enumerate}[\rm (i)]
    \item $sE-DQ$ is a~regular pencil;
    \item $sE-Q$  is a~regular pencil;
    \item For $\Lc=\ran\begin{smallbmatrix}E\\Q\end{smallbmatrix}$, it holds $\dim\Lc=n$, i.e.,\ $\Lc$ is a~self-adjoint linear relation.
\end{enumerate}
Then
\[{\rm (i)}\,\Longrightarrow\,{\rm (ii)}\,\Longleftrightarrow\,{\rm (iii)}.\]
If additionally, $Q^*E\geq 0$ and
\begin{align}
\label{secondcond}
(Q\ker E)\cap \{x\in\ran Q~|\;Dx\in (\ran Q)^\bot\}=\{0\},
\end{align}
then $(ii)\Longrightarrow (i)$.
\end{corollary}
\begin{proof}
By using \eqref{eq:DL-^1rel}, we have that \eqref{factor3} holds for $A=DQ$, $\Dc=\gr D$ and  $\Lc=\ran\begin{smallbmatrix}E\\Q\end{smallbmatrix}$. Then $\Lc$ is symmetric by Lemma~\ref{lem:symequiv}.\\
``(i)\,$\Rightarrow$\,(iii)'': Assume that $sE-DQ$ is regular.  The multi-valued part of $\Dc=\gr D$ is trivial, whence  $\mul\Dc\cap\ker\Lc=\{0\}$. Thus we can apply Proposition~\ref{cor:reg}~(ii), which gives
\[
\ker\Lc=\ker\Lc\widehat+\mul\Dc=(\ker\Lc)^\perp\widehat +(\dom\Dc)^\perp=(\ker\Lc)^\perp.
\]
Then Lemma~\ref{lem:symequiv} yields that $\Lc$ is self-adjoint.\\
``(iii)\,$\Rightarrow$\,(ii)'': Let $\Lc$ be self-adjoint. Then Proposition~\ref{lem:product}~(iv) with $\Dc=-\gr I_n$ implies that there exist unitary matrix $U$ and a~Hermitian matrix $L_{11}$ with 
\begin{equation}
\label{eq:rangetriangedit}\ran\begin{smallbmatrix}E\\Q
\end{smallbmatrix}=\Lc=\ran\diag(U,U)\begin{smallbmatrix}L_{11}&0\\0&I_{n-n_1}\\ D_{11}&0\\D_{21}&0\end{smallbmatrix}
\end{equation}
for some Hermitian $D_{11},L_{11}\in\K^{n_1\times n_1}$ and $D_{21}\in\K^{n_2\times n_1}$ with $D_{11}+D_{11}^*\leq0$. Moreover, by Proposition~\ref{lem:product}~(iv), we further have
\[\ker D_{11}\times \{0\}=\left\{x\in\ran \Lc~|\;Dx\in \ker\Lc\right\}.\]
Since, by Lemma~\ref{lem:dissprop}, $\ran \Lc=(\ker\Lc)^\bot$, we obtain that the latter space is trivial. Therefore,  $D_{11}$ is invertible.
Further, by using \eqref{eq:rangetriangedit}, we obtain that there exists some invertible $T\in\K^{n\times n}$ with
\[
\begin{smallbmatrix}E\\Q\end{smallbmatrix}T=\diag(U,U)\begin{smallbmatrix}L_{11}&0\\0&I_{n-n_1}\\ D_{11}&0\\D_{21}&0\end{smallbmatrix}.
\]
This gives $\det(sE-Q)=\det(UT^{-1})\det(sL_{11}-D_{11})\cdot s^{n-n_1}$. The polynomial $\det(sL_{11}-D_{11})$ is nonzero, since the invertibility of $D_{11}$ yields that it does not vanish at the origin. Therefore, $\det(sE-Q)$ is a~product of nonzero polynomials, whence the pencil $sE-Q$ is regular.\\
``(ii)\,$\Rightarrow$\,(iii)'': If $sE-Q$ is regular, then $\ker E\cap\ker Q=\{0\}$, and the dimension formula gives
\[\dim\Lc=\dim\begin{smallbmatrix}E\\Q\end{smallbmatrix}=n.\]
It remains to prove that ``(ii)\,$\Rightarrow$\,(i)'' holds under the additional assumptions $Q^*E\geq 0$ and \eqref{secondcond}. As we have already shown that (ii) implies (iii), we can further use that $\Lc$ is self-adjoint.
 By using $\Dc=\gr D$, we can apply Proposition~\ref{lem:product}~(iv) to see that there exists a~unitary matrix $U\in\K^{n\times n}$, such that  
 \[ \ran\begin{smallbmatrix}
       E\\A
     \end{smallbmatrix}=
     \Dc\Lc=\ran\diag(U,U)\begin{smallbmatrix}
L_{11}&0\\0& I_{n_2}\\D_{11}& 0\\ D_{21}& 0
\end{smallbmatrix}
\]     
with $n_1=\dim\ran\Lc=\rk Q$, $n_2=n-n_1$, and matrices $L_{ij},D_{ij}\in\K^{n_i\times n_j}$ with $L_{11}=L_{11}^*$ and $D_{11}+D_{11}^*\leq0$.
Invoking \eqref{secondcond}, Proposition~\ref{lem:product}~(iv) further yields that
\begin{multline*}
\{0\}=U^*(Q\ker E)\cap \{x\in\ran Q~ |\;Dx\in (\ran Q)^\bot\}\\=(\ker L_{11}\times\{0\})\cap(\ker D_{11}\times\{0\})=(\ker L_{11}\cap\ker D_{11})\times\{0\},
\end{multline*}
and thus $\ker L_{11}\cap\ker D_{11}=\{0\}$.
On the other hand, the assumption $Q^*E\geq 0$ implies, by using Lemma~\ref{lem:nonneg}, that $\Lc$ is nonnegative. Then Proposition~\ref{lem:product}~(i) implies that $L_{11}\geq0$.
Thus, $sL_{11}-D_{11}$ is positive real, and Lemma~\ref{lem:posreal} together with the already proven identity $\ker L_{11}\cap\ker D_{11}=\{0\}$ yields that $sL_{11}-D_{11}$ is regular. Further, by Lemma~\ref{lem:dissprop} together with the self-adjointness of $\Lc$, we have $\ker \Lc=(\ran\Lc)^\bot$. Additionally invoking $\dom\Dc=\K^n$ and $\mul\Dc=\{0\}$, we see that  $\ker\Lc\widehat+\mul\Dc=(\ran\Lc)^{\perp}\widehat +(\dom \Dc)^{\perp}$. This means that (i) and (ii) in Proposition~\ref{cor:reg} hold, implying that $sE-A$ is regular.\hfill\end{proof}

Note that the statement ``(i)\,$\Rightarrow$\,(ii)'' has already been obtained in \cite[Prop.~4.1]{MehlMehrWojt18}. The implication ``(ii)$\Rightarrow$(i)'' does not hold in general, see \cite[Ex.~4.7]{MehlMehrWojt18}. We present another example which shows that we can construct pencils $sE-DQ$ with arbitrarily large row and column minimal indices.

\begin{example}\label{ex:sing}
Let $n\coloneqq2k+1$, $k\in\N$, and let $Q$ be the identity matrix of size $n\times n$. Further, let $E,D\in\K^{n\times n}$ with
\[
sE-DQ=sE-D=\!\!\begin{bmatrix}
0& -G_k(s)^\top\\ G_k(-s)&0
\end{bmatrix} \text{ and}~  G_k(s)\coloneqq\begin{smallbmatrix}
 1&s&&\\ &\ddots &\ddots &\\&& 1&s\end{smallbmatrix}\in\K[s]^{k\times(k+1)}.
\]
Then we immediately see that $Q^*E=E^*Q$, $D+D^*=0$ and $sE-DQ=sE-D$ is singular. In particular, the pencil has one row and one column minimal index, and both are equal to $k$.
\end{example}

\section{Kronecker form of port-Hamiltonian pencils}
\label{sec:main}

We now investigate the Kronecker structure of port-Hamiltonian pencils. We have seen in Example~\ref{ex:sing} that such pencils may have arbitrarily large row and column indices. On the other hand, the following two examples show that the index and the size of the Jordan blocks on the imaginary axis may be arbitrarily large as well. Note that these examples are furthermore pH in the sense of both \cite{MascvdSc18} and \cite{MehlMehrWojt18}.

\begin{example}
For $k\in\N$, consider the pencil
\[
sL-D=\begin{smallbmatrix}&&&&&-1 \\[-1mm] &&&&\iddots&s \\ &&&-1&\iddots&\\[1mm]&&1&s&&\\ &\iddots&\iddots&&&\\1&s&&&&\end{smallbmatrix}\in\K[s]^{2n\times 2n}
\]
Then $L\in\K^{2n\times 2n}$ is Hermitian and $D\in\K^{2n\times 2n}$ is skew-Hermitian. Hence, the relation $\Dc=\gr D$ is skew-adjoint (in particular dissipative), and $\Lc=(\gr L)^{-1}$ is self-adjoint. Then for $E=L$ and $A=D$, it holds \eqref{factor3}.  It can be seen that $E^{-1}A$ is nilpotent with $(E^{-1}A)^{2n-1}\neq0$. Consequently, the Kronecker form \eqref{KCF} of $sE-A$ is consisting of exactly one Jordan block at the eigenvalue $\infty$ with size $2n$. Therefore, the index of $sE-A$ reads $2n$.
\end{example}

\begin{example}
For $k\in\N$, consider the pencil
\[
sL-D=\begin{smallbmatrix}&&&&&&s \\[-1mm]&&&&&\iddots&-1\\[-1mm] &&&&\iddots&\iddots&\\ &&&s&-1&&\\[-1mm]&&\iddots&1&&&\\[-1mm] &\iddots&\iddots&&&&\\s&1&&&&&\end{smallbmatrix}\in\K[s]^{(2n+1)\times (2n+1)}
\]
which is consisting of the Hermitian matrix $L\in\K^{2n\times 2n}$ and the skew-Hermitian matrix $D\in\K^{2n\times 2n}$. As in the previous example, the choices $\Dc=\gr D$, $\Lc=(\gr L)^{-1}$ lead to the pH pencil $sE-A\coloneqq sL-D$. 
It can be seen that $A^{-1}E$ is nilpotent with $(E^{-1}A)^{2n}\neq0$. Consequently, the Kronecker form \eqref{KCF} of $sE-A$ is consisting of exactly one Jordan block at the eigenvalue $0$ with size $2n+1$.
\end{example}

The previous examples show that additional assumptions on $\Dc$ and $\Lc$ are required for a~further specification of the Kronecker form of pH pencils. In the following, we focus on the case where $\Lc$ is (maximally) nonnegative. Note that the nonnegativity assumption on $\Lc$ has a~physical interpretation in terms of energy functionals \cite{MehlMehrWojt18}. 

From the lower triangular form \eqref{meinblock}, we derive some structural properties of regular pencils $sE-A$ induced by $\ran\begin{smallbmatrix}E\\A\end{smallbmatrix}=\Dc\Lc$ with dissipative $\Dc$ and nonnegative $\Lc$. Besides an index analysis, we will further present some results on the location of the eigenvalues of $sE-A$. We show that $sE-A$ does not have eigenvalues with positive real part and, except for a~possible eigenvalue at the origin of higher order and the purely imaginary eigenvalues are proven to be semi-simple. This corresponds - in a~certain sense - to stability of the system.
\begin{proposition}
\label{cor:mainthm}
Let $E,A\in \K^{n\times n}$ such that $\ran\begin{smallbmatrix}E\\A\end{smallbmatrix}=\Dc\Lc$ for some dissipative relation $\Dc\subset\K^{2n}$ and a nonnegative relation $\Lc\subset\K^{2n}$. If $sE-A$ is regular, then the following holds:
\begin{itemize}
\item[\rm (a)]
$\sigma(E,A)\subseteq \overline{\C_-}$ and the non-zero eigenvalues on the imaginary axis are semi-simple. The size of the Jordan blocks at $0$ is at most two.
\item[\rm (b)] The size of the Jordan blocks at $\infty$, i.e.\ the index, is at most three.
\item[\rm (c)] If additionally $\Dc$ is maximally dissipative  and $\Lc=(\gr L)^{-1}$ for some positive definite  $L\in\K^{n\times n}$, then $sE-A$ has index at most one and the eigenvalue zero is semi-simple.
\end{itemize}
\end{proposition}
\begin{proof}
Since $sE-A$ is regular, Proposition~\ref{cor:reg} yields that there exist invertible $S,T\in\K^{n\times n}$, such that
\begin{align}
\label{lower_proof}
S(sE-A)T=\begin{bmatrix}
s L_{11}- D_{11}&0\\s L_{21}- D_{21}&s L_{22}- D_{22}
\end{bmatrix}\in \K[s]^{n\times n}
\end{align}
with $L_{ij},D_{ij}\in\K^{n_i\times n_j}$ for some $n_1,n_2\in\N$ with $n_1+n_2=n$ and, using Proposition~\ref{lem:product}~(i), we have
\begin{align}
\label{moreproperties}
L_{11}=L_{11}^*\geq 0,\quad  D_{11}+D_{11}^*\leq 0,\quad L_{22}=L_{22}^2=L_{22}^*,\quad -D_{22}=D_{22}^2=D_{22}^*.
\end{align}
and $\ran L_{22}\cap\ran D_{22}=\{0\}$.
It follows from \cite[Thm.~4.1]{ReisStyk11} that
\begin{align}
\label{speknurnull}
\sigma(L_{22},D_{22})\subseteq\{0\}.
\end{align}
and, moreover, the possible eigenvalue zero is semi-simple and the index of $sL_{22}-D_{22}$ is at most one. \\
Further, since $L_{11}\geq0$ and $D_{11}+D_{11}^*\leq0$ implies that $sL_{11}-D_{11}$ is positive real, we have by Lemma~\ref{lem:posreal}, \eqref{speknurnull} and \eqref{lower_proof} that
\[
\sigma(E,A)=\sigma(L_{11},D_{11})\cup\sigma(L_{22},D_{22})\subseteq \overline{\C_-}.
\]
Next we prove (a): As we have already shown that the eigenvalues of $sE-A$ have nonpositive real part, it remains to prove the statements on the sizes of the Jordan blocks of $sE-A$ at $\lambda\in\sigma(E,A)\cap i\R$.
Let $\lambda\in\sigma(E,A)\cap i\R$.
By Lemma~\ref{lem:elementary} we have to show that the order of $\lambda$ as a pole of $(sE-A)^{-1}$ is equal to one, if $\lambda\neq0$, and at most two if $\lambda=0$. We have from \eqref{lower_proof} that 
\begin{align}
&~~~(sE-A)^{-1}\nonumber\\
&=T^{-1}\begin{bmatrix}s L_{11}- D_{11}&0\\s L_{21}- D_{21}&s L_{22}- D_{22}\end{bmatrix}^{-1}S^{-1}\nonumber\\
&=T^{-1}\begin{bmatrix}(s  L_{11}- D_{11})^{-1}&0\\-(s L_{22}- D_{22})^{-1}(s L_{21}- D_{21})(s L_{11}- D_{11})^{-1}&(s L_{22}- D_{22})^{-1}\end{bmatrix}S^{-1}\label{blockinv}
\end{align}
implying that the order of $\lambda$ as a pole of $(sE-A)^{-1}$ is equal to the maximal order of $\lambda$ as a pole of the block entries 
\begin{align}
\label{blockentries}
(sL_{ii}-D_{ii})^{-1},\, i=1,2,\quad \text{and}\quad (s L_{22}- D_{22})^{-1}(s L_{21}- D_{21})(s L_{11}- D_{11})^{-1}.
\end{align}
Since $s L_{11}- D_{11}$ is positive real, the order of $\lambda$ as a pole of $(s L_{11}- D_{11})^{-1}$ is at most one by Lemma~\ref{lem:posreal}. Moreover, by \eqref{speknurnull}, the only possible pole of $(s L_{22}- D_{22})^{-1}$ might be at $\lambda=0$ and this pole is of order one. In summary, this shows that the pole order of \eqref{blockentries} and thus of \eqref{blockinv} at $\lambda=0$ is at most two and the pole order of \eqref{blockinv} at $\lambda\in i\R\setminus\{0\}$ is at most one. This completes the proof of (a).\\ 
We prove (b). Since $s L_{11}- D_{11}$ is positive real, its index is at most two and hence, by Lemma~\ref{lem:elementary} there exist some $M_1,\omega_1>0$ such that
\begin{align}\label{n1}
\forall\,\lambda >\omega_1:\quad 
\|(\lambda  L_{11}- D_{11})^{-1}\|\leq M_1\lambda.
\end{align}
As we have previously shown, the index of $s L_{22}- D_{22}$ is at most one, i.e., there exist some $M_2,\omega_2>0$ such that
\begin{align}
\label{n2}
\forall\,\lambda>\omega_2:\quad 
\|(\lambda  L_{22}- D_{22})^{-1}\|\leq M_2.
\end{align}
A combination of \eqref{n1} and \eqref{n2} yields for all $\lambda>\max\{\omega_1,\omega_2\}$
\begin{align}
\nonumber&\,\,\,\,\,\,\,\,\|(\lambda L_{22}- D_{22})^{-1}(\lambda L_{21}- D_{21})(\lambda L_{11}- D_{11})^{-1}\|\\
&\leq \|(\lambda L_{22}- D_{22})^{-1}\|\|(\lambda L_{21}- D_{21})\|\|(\lambda L_{11}- D_{11})^{-1}\|\label{66}\\
&\leq M_1M_2(\|L_{21}\|+\|D_{21}\|)\lambda^2.\nonumber
\end{align}
Let $M\coloneqq \|S^{-1}\| \|T^{-1}\|M_1M_2(\|L_{21}\|+\|D_{21}\|)|$ and $\omega\coloneqq \max\{\omega_1,\omega_2\}$, then \eqref{66}  implies with \eqref{blockinv} that
\begin{align}
\label{resestim}
\forall\,\lambda>\omega:\quad
\|(\lambda E-A)^{-1}\|\leq M\lambda^{k-1},
\end{align}
with $k=3$ and thus, by Lemma~\ref{lem:elementary}, the index of $sE-A$ is at most three.\\ 
It remains to prove (c). To this end, assume that $\Dc$ is maximally dissipative and that $\Lc=(\gr L)^{-1}$ for some positive definite $L\in\K^{n\times n}$. To show that $sE-A$ has at most index one, we have to verify \eqref{resestim} with $k=1$. Since $L$ is positive definite, Proposition~\ref{lem:product} (i) \& (v) gives $L_{11}\geq 0$ and $\ker L_{11}=\{0\}$. That is, $L_{11}$ is positive definite as well. Hence, we can use \cite[Thm.~4.1]{ReisStyk11} to infer that there exists some $M_3>0$ with
\begin{equation}\label{einsstrich}
\begin{aligned}
\forall\, \lambda>0:\quad \|(\lambda L_{11}-D_{11})^{-1}\|&\leq  \frac{M_3}{\lambda}.
\end{aligned}
\end{equation}
Using \eqref{einsstrich}, there exists some $M_4\coloneqq M_2M_3(\|L_{21}\|+\|D_{21}\|)$ and $\omega_4:\coloneqq\max\{0,\omega_3,\omega_2\}$ such that for all  $\lambda>\omega_4$ it holds 
\begin{align*}
\nonumber&\,\,\,\,\,\,\,\,\|(\lambda L_{22}- D_{22})^{-1}(\lambda L_{21}- D_{21})(\lambda L_{11}- D_{11})^{-1}\|\\
&\leq \|(\lambda L_{22}- D_{22})^{-1}\|\|(\lambda L_{21}- D_{21})\|\|(\lambda L_{11}- D_{11})^{-1}\|\\
&\leq M_2M_3(\|L_{21}\|+\|D_{21}\|)\nonumber\\&=M_4.\nonumber
\end{align*}
Thus, by Lemma~\ref{lem:elementary}, $sE-A$ has index at most one. To conclude that zero is a semi-simple eigenvalue, recall from Proposition~\ref{lem:product} (v) that $D_{22}=-I_{n_2}$, $L_{22}=0$. Consequently, the pole order of \eqref{blockentries} and whence of \eqref{blockinv} at $\lambda=0$ is at most one. As a result of Lemma~\ref{lem:elementary}, the eigenvalue $\lambda=0$ is semi-simple.\hfill
\end{proof}

The following example shows that without maximality assumptions on the subspaces $\Dc$ and $\Lc$ an index of $sE-A$ equal to three is possible.
\begin{example}
Using the canonical unit vectors $e_1,e_2,e_3\in\R^3$ we consider the relations
\[
\Dc=\ran\begin{bmatrix}E_D\\ A_D\end{bmatrix}=\ran\begin{bmatrix}e_1&e_2&0\\-e_2&e_1&e_3\end{bmatrix},\quad \Lc=\ran\begin{bmatrix}E_L\\ A_L\end{bmatrix}=\ran\begin{bmatrix}e_1&e_3\\e_1&e_2\end{bmatrix}.
\]
Since
\[
0=A_D^*E_D+E_D^*A_D\leq 0,\quad  A_L^*E_L=\begin{bmatrix}1&0\\0&0\end{bmatrix}\geq 0, 
\]
we have that 
$\Dc$ is dissipative, and $\Lc$ is nonnegative. It can be further seen that the product of $\Dc$ and $\Lc$ reads
\begin{align*}
\Dc\Lc=\Span\left\{(0,e_3),(e_3,e_1),(e_1,-e_2)\right\},
\end{align*}
and we obtain the range representation \eqref{factor3} with
\[ E\coloneqq\begin{smallbmatrix}0&0&1\\0&0&0\\0&1&0\end{smallbmatrix},\quad A\coloneqq\begin{smallbmatrix}0&1&0\\0&0&-1\\1&0&0\end{smallbmatrix}.
\]
Since $A^{-1}E$ is nilpotent with $(A^{-1}E)^2\neq0$, we have that the Kronecker form of $sE-A$ is consisting of exactly one Jordan block at $\infty$ with size $3$. In particular, the index of $sE-A$ is equal to three. 
\end{example}

Next we show that under the additional assumption that $\Lc$ is the graph of a~positive definite matrix, the pencil $sE-A$ induced by $\Dc\Lc$ is already regular with index one. This result was previously obtained in \cite[Prop.~4.1]{vdSc13} for the special case where $\Dc$ is a skew-adjoint subspace.
\begin{corollary}
\label{cor:index1}
Let $sE-A$ be a matrix pencil with $E,A\in\K^{n\times n}$ and $\ran\begin{smallbmatrix}E\\ A\end{smallbmatrix}=\Dc\Lc$ and let $\Dc\subseteq \K^{2n}$ be maximally dissipative and $\Lc=(\gr Q)^{-1}$ for some positive definite $Q\in\K^{n\times n}$. Then $sE-A$ is regular and has index at most one. 
\end{corollary}
\begin{proof}
Since $\Lc=(\gr Q)^{-1}=\gr(Q^{-1})$ we have $\mul\Dc\cap\ker\Lc=\mul\Dc\cap\{0\}=\{0\}$ and by 
 Proposition~\ref{lem:product}~(v) there exist unitary $U,X\in\K^{n\times n}$ such that 
\begin{align}
\label{nochmal_lower_triag}
U^*(sE-A)X=\begin{smallbmatrix}sL_{11}-D_{11}&0\\sL_{21}&I_n\end{smallbmatrix},
\end{align}
with $sL_{11}-D_{11}$ positive real and $\ker L_{11}\times \{0\}=\,U^*\left\{x\in\dom\Dc~|\;Qx\in \mul\Dc\right\}$. 
Hence, if $x\in \ker L_{11}\times \{0\}$, then
$x\in\dom\Dc$ with $Qx\in\mul\Dc$. In virtue of Lemma~\ref{lem:dissprop}, we have $\mul\Dc=(\dom\Dc)^\perp$ and hence $\langle Qx,x\rangle=0$, and the positive definiteness of $Q$ leads to $x=0$. Consequently, the kernel of $L_{11}$ is trivial, and we obtain $\ker L_{11}\cap\ker D_{11}=\{0\}\cap\ker D_{11}=\{0\}$. Now invoking Lemma~\ref{lem:posreal}\,(a), we obtain that $sL_{11}-D_{11}$ is regular and thus, by \eqref{nochmal_lower_triag}, $sE-A$ is regular, too. Moreover, the index is at most one by Proposition~\ref{cor:mainthm}\,(c).\hfill
\end{proof}

The main result on the Kronecker form of port-Hamiltonian DAEs is given below. Here we additionally assume the maximality of the underlying subspaces. 
\begin{theorem}
\label{thm:singular}
Let $E,A\in\K^{n\times m}$ such that $\ran\begin{smallbmatrix}
E\\ A \end{smallbmatrix}=\Dc\Lc$ for some maximally dissipative relation  $\Dc\subseteq \K^{2n}$ and a~maximally nonnegative relation $\Lc\subseteq \K^{2n}$. Then there exist invertible $S\in\K^{n\times n}$, $T\in\K^{m\times m}$ and $n_i\in\N,i=1,2,3,4$, such that
\begin{align}
\label{maximalkcf_vorne}
S(sE-A)T=\begin{bmatrix}
s\tilde L_{11}-\tilde D_{11}&0&0&0&0&0\\ \tilde D_{21}&sI_{n_2}&0&0&0&0\\s\tilde L_{21}&0&I_{n_3}&0&0&0\\0&0&0&sI_{n_4}&-I_{n_4}&0\\0&0&0&0&0&0
\end{bmatrix},
\end{align}
where $s\tilde L_{11}-\tilde D_{11}\in\K[s]^{n_1\times n_1}$ is regular and positive real and $\ker \tilde L_{11}\subset\ker \tilde L_{21}$.

In particular, the Kronecker form of $sE-A$ has the following properties:
\begin{itemize}
    \item[\rm (a)] The column minimal indices are at most one (if there are any).

    \item[\rm (b)] The row minimal indices are zero (if there are any).

    \item[\rm (c)]  We have  $\sigma(E,A)\subseteq\overline{\C_-}$. Furthermore, the non-zero eigenvalues on the imaginary axis are semi-simple. The Jordan blocks at $\infty$ and at zero have size at most two, i.e.\ the index is at most two.
\end{itemize}
\end{theorem}
\begin{proof}
A proof of the block diagonal decomposition \eqref{maximalkcf_vorne} with
positive real $s\tilde L_{11}-\tilde D_{11}\in\K[s]^{n_1\times n_1}$ and $\ker \tilde L_{11}\subset\ker \tilde L_{21}$ is given in  Proposition~\ref{prop:kron} in the appendix. First observe that the block lower-triangular pencil
\begin{align}
\label{regpart}
sE_r-A_r\coloneqq\begin{bmatrix}
s\tilde L_{11}-\tilde D_{11}&0&0\\ \tilde D_{21}&sI_{n_2}&0\\s\tilde L_{21}&0&I_{n_3}
\end{bmatrix}
\end{align}
obtained from \eqref{maximalkcf_vorne}
is regular. Since, moreover, a~simple column permutation yields that the Kronecker form of $[sI_{n_4},-I_{n_4}]$ is given by $\diag(sK_2-L_2,\ldots,sK_2-L_2)\in\K[s]^{n_4\times 2n_4}$, we obtain that the column minimal indices of $sE-A$ are one (if there are any) and the row minimal indices of $sE-A$ are at most zero (if there are any). This proves (a) \& (b).\\
We continue with the proof of (c). Considering \eqref{maximalkcf_vorne}, \eqref{regpart} and invoking Lemma~\ref{lem:posreal} (c) yields
\[
\sigma(E,A)=\sigma(E_r,A_r)\subseteq \sigma(\tilde L_{11},\tilde D_{11})\cup\{0\}\subseteq\overline{\C_-}.
\]
It remains to show the statements on the index and the sizes of the Jordan blocks to eigenvalues on the imaginary axis. Here we proceed as in the proof of  Proposition~\ref{cor:mainthm} by using the resolvent of \eqref{regpart} which is given by
\begin{align}
\label{diefct}
\begin{bmatrix}
 s\tilde L_{11}-\tilde D_{11}&0&0\\ \tilde D_{21}&s I_{n_2}&0\\s\tilde L_{21}&0&I_{n_3}
\end{bmatrix}^{-1}=\begin{bmatrix}(s\tilde L_{11}-\tilde D_{11})^{-1}&0&0\\-s^{-1}\tilde D_{21}(s\tilde L_{11}-\tilde D_{11})^{-1}&s^{-1}I_{n_2}&0\\-s \tilde L_{21}(s\tilde L_{11}-\tilde D_{11})^{-1}&0&I_{n_3}\end{bmatrix}.
\end{align}
Regarding Lemma~\ref{lem:elementary}, the pole order of \eqref{diefct} at $\lambda\in\sigma(E,A)$ is equal to the size of the largest Jordan block of \eqref{regpart} at $\lambda$. Since $s \tilde L_{11}-\tilde D_{11}$ is positive real, the pole order of \eqref{diefct} at the non-zero eigenvalues on the imaginary axis is at most one and hence these eigenvalues are semi-simple. The pole order of $(sE_r-A_r)^{-1}$ at $\lambda=0$ is at most two and hence the size of the Jordan blocks at $0$ in the Kronecker form of $sE-A$ is at most two, by Lemma~\ref{lem:elementary}. 

We finally show that the index of $sE-A$ as in \eqref{def:index} is at most two. Since the index is invariant under pencil equivalence of $sE_r-A_r$ we can assume without restriction that $s\tilde L_{11}-\tilde D_{11}$ is already given in Weierstra\ss\ canonical form. Further, $s\tilde L_{11}-\tilde D_{11}$ is positive real and hence its the index is by Lemma~\ref{lem:posreal}~(d) at most two. Altogether, we obtain for some $k_1,k_2\in\N$ and $\tilde J\in\K^{k_2\times k_2}$ in Jordan canonical form that
\begin{align}
\label{l11_wcf}
s\tilde L_{11}-\tilde D_{11}=\diag\left(\begin{smallbmatrix}-1&s\\0&-1\end{smallbmatrix},\ldots,\begin{smallbmatrix}-1&s\\0&-1\end{smallbmatrix},-I_{k_1},sI_{k_2}-\tilde J\right).
\end{align}
Consequently, there exist $M_1,\omega_1>0$ such that 
\begin{align}
\label{l11_res_estim}
\forall \lambda >\omega_1:\quad \|(\lambda\tilde L_{11}-\tilde D_{11})^{-1}\|\leq M_1\lambda.
\end{align}
Looking at the block entries of \eqref{diefct}, we continue to show the existence of some $M_2,\omega_2>0$ satisfying
\begin{align}
\label{l11estim}
\forall \lambda>\omega_2:\quad \|\lambda\tilde L_{21}(\lambda \tilde L_{11}-\tilde D_{11})^{-1}\|\leq M_2 \lambda.
\end{align}
Invoking the block diagonality of $s\tilde L_{11}-\tilde D_{11}$ and the structure of the blocks in \eqref{l11_wcf} it suffices to show that \eqref{l11estim} holds for $s\tilde L_{11}-\tilde D_{11}=\begin{smallbmatrix}-1&s\\0&-1\end{smallbmatrix}$.
Proposition~\ref{prop:kron} yields $\ker \tilde L_{11}\subset\ker \tilde L_{21}$, which implies with $\ker \tilde L_{11}=\{\alpha e_1 \,|\, \alpha\in\K\}$ for $x=\begin{smallpmatrix}x_1\\x_2\end{smallpmatrix}\in\K^2$ and for all $\lambda>0$ and $M_2\coloneqq\|\tilde L_{21}e_1\|$ that
\begin{align*}
\|\lambda\tilde L_{21}(\lambda \tilde L_{11}-\tilde D_{11})^{-1}x\|&=\left\|\lambda\tilde L_{21}\begin{bmatrix}-1&-\lambda\\0&-1\end{bmatrix}\begin{pmatrix}x_1\\x_2\end{pmatrix}\right\|\\&=\left\|\lambda\tilde L_{21}\begin{pmatrix}-x_1-\lambda x_2\\-x_2\end{pmatrix}\right\|\\&=\|-\lambda \tilde L_{21}e_2x_2\|\\&\leq M_2\lambda\|x\|.
\end{align*}
This proves \eqref{l11estim}. From \eqref{diefct} together with \eqref{l11_res_estim} and \eqref{l11estim}, we see that there exist some $M,\omega>0$ with 
\begin{align}
\label{indextwo}
\forall \lambda>\omega:\quad \|(\lambda E_r-A_r)^{-1}\|\leq M\lambda.
\end{align}
This means by Lemma~\ref{lem:elementary} that $\alpha_i\leq 2$ for all $i=1,\ldots,\ell_\alpha$. Furthermore, the block structure in \eqref{maximalkcf_vorne} implies $\gamma_i\leq1$ for all $i=1,\ldots,\ell_{\gamma}$ and hence the index of $sE-A$ as in \eqref{def:index} is at most two.\hfill
\end{proof}

The following example from \cite{MehlMehrWojt18} shows that without the maximality assumption on $\Lc$, arbitrarily large row minimal indices might occur.
\begin{example}
\label{ex:MMW}
Let $\Dc=\gr D$, $D=J_n(0)-J_n(0)^*$ where $J_n(0)\in\R^{n\times n}$ is a~Jordan block at $0$ and $\Lc=\ran\begin{smallbmatrix}E\\Q\end{smallbmatrix}$ for $E=Q=[I_{n-1}\,,\,0_{(n-1)\times 1}]^*$. Then $\Lc$ is nonnegative, but not maximal. Then, for $A=DQ$, it holds \eqref{factor3}, and it is shown in \cite{MehlMehrWojt18} that the pencil $sE-A$ has one row minimal index equal to $n-1$.
\end{example}

We give a~brief comparison of Theorem~\ref{thm:singular} with \cite[Thm.~4.3]{MehlMehrWojt18}, where pH pencils in the sense of \cite{MehlMehrWojt18} with, additionally, $Q^*E\geq 0$ are considered.
\begin{remark}
\begin{itemize}
\item[\rm (i)] As \cite[Thm.~4.3]{MehlMehrWojt18} treats pH pencils in the sense of \cite{MehlMehrWojt18}, it employs the assumption that $\mul\Dc=\{0\}$.
\item[\rm (ii)]  \cite[Thm.~4.3]{MehlMehrWojt18} shows that pH pencils in the sense of \cite{MehlMehrWojt18} have the property that all its eigenvalues have nonpositive real part. Further, the nonzero imaginary eigenvalues are semi-simple. 
A statement on the sizes of the Jordan blocks corresponding to the eigenvalue zero is not contained.
\item[\rm (iii)] Instead of our assumption of maximality of the nonnegative relation $\Lc=\ran\begin{smallbmatrix}
E\\Q
\end{smallbmatrix}$, the weaker assumption that all row minimal indices of $sE-Q$ are zero has been used in \cite[Thm.~4.3]{MehlMehrWojt18} to describe the Kronecker form of pencils which are pH in the sense of \cite{MehlMehrWojt18}.
\end{itemize}
\end{remark}

We present an example of a~pencil which is subject of Theorem~\ref{thm:singular} but it cannot be represented as a~pencil which is subject of \cite[Thm.~4.3]{MehlMehrWojt18}.
\begin{example}
Let $E=\begin{smallbmatrix}
1&0\\1&0
\end{smallbmatrix}$, $A=\begin{smallbmatrix}
-1&0\\0&-1
\end{smallbmatrix}$ and consider
\[
\Dc=\ran\begin{smallbmatrix}1&0\\0&0\\-1&0\\0&1
\end{smallbmatrix},\quad \Lc=\left(\gr\begin{bmatrix}1&1\\1&1
\end{bmatrix}\right)^{-1}.
\]
Then $\Dc$ is maximally dissipative, $\Lc$ is maximally nonnegative, and
$\ran\begin{smallbmatrix}
E\\A
\end{smallbmatrix}=\Dc\Lc$. Therefore, the pencil $sE-A$ meets the assumptions of Theorem~\ref{thm:singular}.\\
We show in the following that it is not possible to rewrite $\Dc\Lc=(\gr D)\hat \Lc$ for some dissipative matrix $D\in\K^{2\times 2}$ and a nonnegative relation $\hat\Lc\subset\K^4$. To this end, let $\hat \Lc=\ran\begin{smallbmatrix}
\hat E\\\hat Q
\end{smallbmatrix}$ with $\hat Q^*\hat E\geq 0$. Then
\[
\ran\begin{smallbmatrix}
E\\A
\end{smallbmatrix}=(\gr D)\,\, \ran\begin{smallbmatrix}
\hat E\\\hat Q
\end{smallbmatrix}= \ran\begin{smallbmatrix}
\hat E\\D\hat Q
\end{smallbmatrix}
\]
and hence there exists some invertible $T\in\K^{2\times 2}$ with $\hat ET=E$ and $D\hat QT=A$. Thus $D\hat QT=-I_2$ and hence $\hat QT=-D^{-1}$.
With $\hat QT=\begin{smallbmatrix} q_1&q_2\\q_3&q_4\end{smallbmatrix}$ we have $T^*\hat Q^*E=\begin{smallbmatrix}q_1+q_3&0\\q_2+q_4&0\end{smallbmatrix}\geq 0$ and hence $q_1+q_3\geq 0$ and $q_2+q_4=0$. Since $D$ is dissipative, $\hat QT$ is also dissipative and therefore
\[
0\geq \langle\begin{smallpmatrix}1\\1\end{smallpmatrix},(D+D^*)\begin{smallpmatrix}1\\1\end{smallpmatrix}\rangle=2\re\langle\begin{smallpmatrix}1\\1\end{smallpmatrix},D\begin{smallpmatrix}1\\1\end{smallpmatrix}\rangle=\re (q_1+q_2+q_3+q_4)=q_1+q_3\geq 0.
\]
This implies $q_1+q_3=0$ and hence $\begin{smallpmatrix}1\\1\end{smallpmatrix}\in\ker(\hat QT)^*=\ker\hat Q^*$, which contradicts the invertibility of $\hat Q$.
\end{example}

\section{Appendix}
In this part we present the proof of  Proposition~\ref{lem:product}. After that, we present Proposition~\ref{prop:kron}, which is an essential ingredient for the proof of Theorem~\ref{thm:singular}. Note that in these proofs we use the already proven results presented prior to Proposition~\ref{lem:product}, whereas the proof of Proposition~\ref{prop:kron} will make use of Proposition~\ref{lem:product}.\\
We will use the following notation throughout the proofs:
If two linear relations $\Lc,\cM\subset\K^{2n}$ are orthogonal, we write $\Lc\widehat\oplus\cM$ for their direct componentwise sum.
If $\Lc,\cM\subset\K^{2n}$ fulfill $\Lc\subseteq\cM$, the \textit{orthogonal minus} is given by  $\cM\widehat\ominus\Lc\coloneqq\cM\cap\Lc^\perp$.
Further, for a~subspace $X\subset\K^n$, the \textit{orthogonal projector} onto $X$ is denoted by $P_X$. For spaces $Y_1,Y_2,Y_3\subset \K^n$ with $Y_1\subset Y_2$ and a~linear operator $M:Y_2\to Y_3$, $M|_{Y_2}$ denotes the restriction of $M$ to the space $Y_2$.

{\em Proof of Proposition~\ref{lem:product}.}\\ 
\emph{Step 1:} We show that there exist orthogonal decompositions
\begin{align}
\label{dl_zerl}
\Dc=\{(x,Dx)\}\widehat\oplus(\{0\}\times\mul\Dc),\quad \Lc=\{(Lx,x)\}\widehat\oplus(\ker\Lc\times\{0\})
\end{align}
for linear operators $D:\dom\Dc\rightarrow (\mul\Dc)^\perp$ and $L:\ran\Lc\rightarrow (\ker\Lc)^\perp$. The result is proved only for $\Dc$; the statement for $\Lc$ is analogous. Consider the operator $D$ with $Dx=P_{(\mul\Dc)^\perp}y$ for $(x,y)\in\Dc$. To show that $D:\dom\Dc\rightarrow (\mul\Dc)^\perp$ is well-defined, let $(x,y),(x,z)\in\Dc$, then $(0,y-z)\in\Dc$ implying that $y-z\in\mul\Dc$. Consequently, $P_{(\mul\Dc)^\perp}y-P_{(\mul\Dc)^\perp}z=P_{(\mul\Dc)^\perp}(y-z)=0$.  Then the equality for the subspace $\Dc$ in \eqref{dl_zerl} follows immediately and, by construction, the summands are orthogonal. 
\emph{Step 2:} We show that
\begin{align}
\label{dlfactor}
\Dc\Lc=\left(\begin{bmatrix} L\\ D\end{bmatrix}(\dom\Dc\cap\ran\Lc)\right)\widehat\oplus\left(\ker \Lc\times\{0\}\right)\widehat\oplus\left(\{0\}\times \mul\Dc\right).
\end{align}
To prove ``$\subseteq$'', let $(x,z)\in\Dc\Lc$. Then there exists some $y\in\K^n$ such that  $(x,y)\in\Lc$ and $(y,z)\in\Dc$. Therefore, $y\in\ran\Lc\cap\dom\Dc$. This implies with \eqref{dl_zerl} that $x=Ly+v_L$ and $z=Dy+v_D$ for some $v_L\in\ker\Lc$ and $v_D\in\mul\Dc$. Hence, \[(x,z)\in\left(\begin{bmatrix}L\\ D\end{bmatrix}(\ran\Lc\cap\dom\Dc)\right)\widehat\oplus\left(\ker \Lc\times\{0\}\right)\widehat\oplus\left(\{0\}\times \mul\Dc\right).\]
To prove ``$\supseteq$'', let $(Ly+v_L,Dy+v_D)\in\K^{2n}$ with $y\in\ran\Lc\cap\dom\Dc$, $v_L\in\ker\Lc$, and $v_D\in\mul\Dc$. This implies $(Ly,y)\in\Lc
$, $(y,Dy)\in\Dc$ and hence $(Ly,Dy)\in\Dc\Lc$. Then $(0,0)\in\Dc$ and $(0,0)\in\Lc$ further lead to $(v_L,0),(0,v_D)\in\Dc\Lc$, and thus $(Ly+v_L,Dy+v_D)\in\Dc\Lc$.\\
\emph{Step 3:} Consider the orthogonal decomposition $\K^n=X_1\widehat\oplus X_2$ with
\begin{align}
\label{Kn_zerl}
X_1\coloneqq\ran\Lc\cap\dom\Dc,\quad  X_2\coloneqq(\ran\Lc\cap\dom\Dc)^\perp=(\ran\Lc)^\perp \widehat+(\dom\Dc)^\perp.
\end{align}
Our next objective is to show
\begin{align}
\label{mulranincl2}
\left(\ker \Lc\times\{0\}\right)\widehat\oplus\left(\{0\}\times \mul\Dc\right)=\begin{bmatrix}P_{\ker\Lc}\\-P_{\mul\Dc}\end{bmatrix}(\ker\Lc\widehat \dotplus\mul\Dc).
\end{align}
The inclusion ``$\supseteq$'' in \eqref{mulranincl2} is immediate.  To prove   ``$\subseteq$'', 
it suffices to show 
that both spaces $\ker\Lc\times \{0\}$ and $\{0\}\times \mul\Dc$ are contained in the set on right hand side of  \eqref{mulranincl2}.
Consider the space $X_3\coloneqq\ker\Lc\widehat \dotplus\mul\Dc$. Then by Lemma~\ref{lem:dissprop} we have  $\ker\Lc\subseteq(\ran\Lc)^\perp$ and $\mul\Dc\subseteq(\dom\Dc)^\perp$, whence $X_3\subseteq X_2$. Since $X_3\widehat\ominus\mul\Dc\subset \ker\Lc$, we have
$(\ker\Lc)^\perp\cap(X_3\widehat\ominus\mul\Dc)=\{0\}$, we have that
$P_{\ker\Lc}|_{X_3\widehat\ominus\mul\Dc}$ is injective. 
This together with $\dim(X_3\widehat\ominus\mul\Dc)=\dim\ker\Lc$ gives $P_{\ker\Lc}(X_3\widehat\ominus\mul\Dc)=\ker\Lc$.
Hence, for each $(v_L,0)\in\ker\Lc\times\{0\}$ there exists $x\in X_3\widehat\ominus\mul\Dc$ with $P_{\ker\Lc}x=v_L$ and $P_{\mul\Dc}x=0$ and therefore $(v_L,0)\in\begin{smallbmatrix}P_{\mul\Lc}\\-P_{\mul\Dc}\end{smallbmatrix}(X_3)$. Analogously, we can show that $\{0\}\times \mul\Dc\subseteq\begin{smallbmatrix}P_{\mul\Lc}\\-P_{\mul\Dc}\end{smallbmatrix}(X_3)$, which altogether shows \eqref{mulranincl2}.\\ 
\emph{Step 4:} Based on the space decomposition $\K^n=X_1\widehat\oplus X_2$ as in \eqref{Kn_zerl}, we define
\begin{align}
\label{lhats}
\hat L_{11}\coloneqq P_{X_1}L|_{X_1},\quad \hat L_{21}\coloneqq P_{X_2}L|_{X_1},\quad  \hat L_{22}\coloneqq P_{\ker\Lc}:X_2\rightarrow X_2
\end{align}
and
\[
\hat D_{11}\coloneqq P_{X_1}D|_{X_1},\quad \hat D_{21}\coloneqq P_{X_2}D|_{X_1},\quad  \hat D_{22}\coloneqq-P_{\mul\Dc}:X_2\rightarrow X_2.
\]
Let $n_i\coloneqq\dim X_i$, $i=1,2$, and
$U_1\coloneqq[u_1,\ldots,u_{n_1}]\in\K^{n\times n_1}$ and $U_2\coloneqq[u_{n_1+1},\ldots,u_n]\in\K^{n\times n_2}$, where the columns are an orthonormal basis of $X_1$ and $X_2$, respectively. Then $U=[U_1,U_2]\in\K^{n\times n}$ is unitary and 
\begin{align}
L_{ij}\coloneqq U_i^*\hat L_{ij}U_j,\quad D_{ij}\coloneqq U_i^*\hat D_{ij}U_j,\quad i,j=1,2.
\end{align}
Combining \eqref{dlfactor} and \eqref{mulranincl2}, we obtain
\begin{align*}
\Dc\Lc&=\left(\begin{smallbmatrix} L\\ D\end{smallbmatrix}(X_1)\right)\widehat\oplus\left(\ker \Lc\times\{0\}\right)\widehat\oplus\left(\{0\}\times \mul\Dc\right)\\
&=\begin{smallbmatrix} \hat L_{11}\\\hat L_{21}\\ \hat D_{11}\\\hat D_{21} \end{smallbmatrix}(X_1)\widehat\oplus\begin{smallbmatrix}0\\\hat L_{22}\\0\\\hat D_{22}\end{smallbmatrix}(X_2)\\
&=\diag(U,U)\left(\begin{smallbmatrix}  L_{11}&0\\ L_{21}&0\\  D_{11}&0\\ D_{21}&0 \end{smallbmatrix}\underbrace{(U^*X_1)}_{=\K^{n_1}\times\{0\}}\widehat\oplus\begin{smallbmatrix}0&0\\0& L_{22}\\0&0\\ 0&D_{22}\end{smallbmatrix}\underbrace{(U^*X_2)}_{=\{0\}\times\K^{n_2}}\right)\\
&=\diag(U,U)\;\ran\begin{smallbmatrix}
L_{11}&0\\L_{21}&L_{22}\\
D_{11}&0\\
D_{21}&D_{22}
\end{smallbmatrix}.
\end{align*}
This completes the proof of \eqref{dl_final}.\\
\emph{Step 5:} We show that \eqref{L11D11_prop} and \eqref{l22d22_prop} hold. 
Let $(y,x)\in \Lc$. Then $y=Lx+v_L$ for some $v_L\in\ker\Lc\subseteq(\ran\Lc)^\perp$ and some $x\in X_1$. Consequently, 
\begin{align}
\label{symmrechnung}
\langle \hat L_{11}x,x\rangle=\langle P_{X_1}L x,x\rangle=\langle L x,x\rangle=\langle L x+v_L,x\rangle=\langle y,x\rangle=\langle x,y\rangle=\langle x,\hat L_{11}x\rangle,
\end{align}
where in the second last equation the symmetry of $\Lc$ was used and the last equation follows from a repetition of the first steps in the second component of the inner product. This implies that $\hat L_{11}$ is Hermitian. Consequently, $L_{11}=U_1^*\hat L_{11}U_1$ is Hermitian. Similarly, one can show that if $\Dc$ is dissipative then $D_{11}$ is dissipative, whence \eqref{L11D11_prop} holds.
Since $L_{22}=U_2^*\hat L_{22}U_2$ and $D_{22}=U_2^*\hat D_{22}U_2$ with orthogonal projectors $\hat L_{22}=P_{\ker\Lc}$ and $-\hat D_{22}=P_{\mul \Dc}$ we have
\begin{align*}
L_{22} &=U_2^*\hat L_{22}U_2 =U_2^*\hat L_{22}^2U_2=U_2^*\hat L_{22}U_2U_2^*\hat L_{22}U_2 =L_{22}^2 =L_{22}^*,\\
-D_{22} &=U_2^*\hat D_{22}U_2 =U_2^*\hat D_{22}^2U_2 =U_2^*\hat D_{22}U_2U_2^*\hat D_{22}U_2 =D_{22}^2 =-D_{22}^*.
\end{align*}
Furthermore, 
\[
\ran D_{22}\cap\ran L_{22}=U_2^*(\ran P_{\mul\Dc}\cap\ran P_{\ker\Lc})=U_2^*(\mul\Dc\cap\ker\Lc)=\{0\},
\]
which implies $\mul\Dc\cap\ker\Lc=\{0\}$ and hence \eqref{l22d22_prop}. \\
\emph{Step 6:}
We prove (i)-(iii). If $\Lc$ is nonnegative, then $\langle y,x\rangle\geq 0$ for all $(x,y)\in\Lc$ which implies, by using \eqref{symmrechnung}, that $\langle \hat L_{11}x,x\rangle\geq 0$ for all $x\in X_1$ and thus $L_{11}=U_1^*\hat L_{11}U_1$ is positive semi-definite. 
Next we show that $\ker L_{11}\subset \ker L_{21}$, if $\Lc$ is maximal. From the maximality we have $(\ker \Lc)^\perp=\ran\Lc$ and thus the operator $L:\ran\Lc\rightarrow\ran\Lc$ from Step~1 can be decomposed as
\[
 L=\begin{bmatrix} \hat L_{11}&\tilde L_{21}^*\\\tilde L_{21}&\tilde L_{22}\end{bmatrix},\quad \ran\Lc=(\dom\Dc\cap\ran\Lc)\widehat\oplus(\ran\Lc\widehat\ominus (\dom\Dc\cap\ran\Lc)),
\]
and $L$ is nonnegative, i.e., $\langle Lx,x\rangle\geq 0$ for all $x\in\ran\Lc$. We show that $\ker \hat L_{11}\subset\ker\tilde L_{21}$. 
Assume that there exists some $x\in\ker\hat L_{11}$ with $z\coloneqq- \tilde L_{21}x\neq 0$.
Since $L\geq 0$ we have for all $\alpha\in\R$
\[
0\leq \left\langle L\begin{pmatrix}\alpha x\\z\end{pmatrix},\begin{pmatrix}\alpha x\\z\end{pmatrix}\right\rangle=\left\langle \begin{bmatrix} \hat L_{11}&\tilde L_{21}^*\\\tilde L_{21}&\tilde L_{22}\end{bmatrix}\begin{pmatrix}\alpha x\\z\end{pmatrix},\begin{pmatrix}\alpha x\\z\end{pmatrix}\right\rangle
=-2\alpha\|z\|^2+\| \tilde L_{22}z\|^2.
\]
Choosing $\alpha$ sufficiently large, we obtain a~contradiction. Hence $\ker \hat L_{11}\subset\ker\tilde L_{21}$. Further, decompose $X_2=(X_2\cap\ran\Lc)\widehat\oplus ( X_2\cap(\ran\Lc)^\perp)$ and, without restriction, assume that the vectors $u_{n_1+1},\ldots,u_{n_1+\hat k}$ for some $\hat k\geq 1$ are an orthonormal basis of $X_2\cap\ran\Lc$. Then 
\[\hat L_{21}=P_{X_2}L|_{X_1}=P_{X_2\cap\ran\Lc}L\vert_{X_1}+P_{X_2\cap(\ran\Lc)
^\perp}L\vert_{X_1}=P_{X_2\cap\ran\Lc}L\vert_{X_1}=\tilde{L}_{21}
\]
and this implies \begin{multline*}
    \ker L_{11}=\ker U_1^*\hat L_{11}U_1=U_1^*\ker \hat L_{11}\\\subset U_1^*\ker \tilde L_{21}=\ker  U_1^*\hat  L_{21}=\ker U_2^*\hat L_{21}U_1=\ker L_{21}.
    \end{multline*}

The assertion (ii) can be proven analogously to (i).
To show (iii), first assume that $\ker L_{22}\cap\ker D_{22}=\{0\}$. Then 
\begin{align}
\label{kerl22_hatl22}
    \ker \hat L_{22}\cap \ker \hat D_{22}=U_2(\ker L_{22}\cap\ker D_{22})=\{0\}
\end{align}
and taking orthogonal complements in $X_2$, we obtain
\[
X_2=(\ker \hat L_{22}\cap\ker\hat D_{22})^\perp=\ran \hat L_{22}\widehat +\ran \hat D_{22}=\ker\Lc\widehat\dotplus\mul\Dc.
\]
Conversely, assume that $X_2=\ker\Lc\widehat\dotplus\mul\Dc$. Then, again by taking orthogonal complements in $X_2$,
\[
\ker \hat L_{22}\cap\ker\hat D_{22}=(\ker\Lc\widehat\dotplus\mul\Dc)^\perp=X_2^\perp=\{0\}.
\]
Now invoking \eqref{kerl22_hatl22} and the injectivity of $U_2$, we obtain $\ker L_{22}\cap\ker D_{22}=\{0\}$.\\
\noindent
\emph{Step 7:} We prove (iv). Assume that $\Lc$ is self-adjoint and $\Dc=\gr D$ for some dissipative $D\in\K^{n\times n}$. Then we have that $\mul\Dc=\{0\}=(\dom\Dc)^\perp$ and $\ker\Lc=(\ran\Lc)^\perp$. Hence,  $X_1=\ran\Lc=X_2^\perp$. This implies that $\hat L_{21}=\hat D_{22}=0$ and thus $L_{21}=D_{22}=0$. Invoking (iii), we have $\ker L_{22}=\ker L_{22}\cap\ker D_{22}=\{0\}$ which implies $L_{22}=I_{n_2}$. Furthermore,  $\mul \Lc=\ker\hat L_{11}=U(\ker L_{11}\times\{0\})$ and together with $(\ran\Lc)^\perp=\ker\Lc$ we obtain
\[
\{x\in\ran\Lc ~|~Dx\in(\ran\Lc)^\perp\}=\ker (P_{\ran\Lc}D|_{\ran\Lc})=\ker\hat D_{11}=U(\ker D_{11}\times \{0\}).
\]
The proof of (v) is analogous to the proof of (iv) and is therefore omitted.
\hfill$\qed$

\begin{proposition}
\label{prop:kron}
Let $\Dc\subseteq\K^{2n}$ be maximally dissipative and $\Lc\subseteq\K^{2n}$ be maximally nonnegative. Further, let $E,A\in\K^{n\times m}$ be such that  $\ran\begin{smallbmatrix}
E\\A
\end{smallbmatrix}=\Dc\Lc$. Then there exist some invertible $S\in\K^{n\times n}$, $T\in\K^{m\times m}$
and $n_i\in\N,i=1,2,3,4$, such that
\begin{align}
\label{maximalkcf_hinten}
S(sE-A)T=\begin{bmatrix}
sL_{11}-D_{11}&0&0&0&0&0\\D_{21}&sI_{n_2}&0&0&0&0\\sL_{21}&0&I_{n_3}&0&0&0\\0&0&0&sI_{n_4}&-I_{n_4}&0\\0&0&0&0&0&0
\end{bmatrix},
\end{align}
where $s L_{11}- D_{11}\in\K[s]^{n_1\times n_1}$ is regular and positive real and $\ker L_{11}\subset\ker L_{21}$.
\end{proposition}
\begin{proof}
The proof consists of two steps. In the first step we derive a~certain range representation for $\Dc\Lc$. In second step, \eqref{maximalkcf_hinten} is obtained from the resulting range representation.\\
{\em Step 1:} We show that there exists some $\hat m\in\N$ and an invertible matrix $S\in\K^{n\times n}$ and some $n_1,n_2,n_3,n_4\in\N$, such that
\begin{equation}
\label{finalpencil}
\allowdisplaybreaks\begin{aligned}
\Dc\Lc=&\diag(S,S)\;\ran\begin{smallbmatrix}
L\\D
\end{smallbmatrix}\\
sL-D=&\begin{smallbmatrix}
sL_{11}- D_{11}&0&0&0&0&0\\D_{21}&sI_{n_2}&0&0&0&0\\sL_{21}&0&I_{n_3}&0&0&0\\0&0&0&sI_{n_4}&-I_{n_4}&0\\0&0&0&0&0&0
\end{smallbmatrix}\in\K[s]^{n\times \hat m}
\end{aligned}
\end{equation}
for some positive real and regular pencil $s  L_{11}-  D_{11}\in\K[s]^{n_1\times n_1}$, $ D_{21}\in\K^{n_2\times n_1}$, $ L_{21}\in\K^{n_3\times n_1}$.
\\
Consider the space $X\coloneqq\mul\Dc\cap\ker\Lc$,
and the relations
\[
\hat\Dc\coloneqq\Dc\widehat\ominus(\{0\}\times X),\quad \hat\Lc\coloneqq\Lc\widehat\ominus(X\times \{0\}).
\]
Then
we obtain an orthogonal decomposition
\begin{align}
\label{tildeohnetilde}
\Dc\Lc=\hat\Dc\hat\Lc\widehat\oplus(\{0\}\times X)\widehat\oplus(X\times \{0\})
\end{align}
and
\begin{align*}
\mul\hat\Dc =\mul\Dc\widehat\ominus X,\quad  \ker\hat\Lc =\ker\Lc\widehat\ominus X.
\end{align*}
This implies $\mul\hat\Dc\cap\ker\hat\Lc=\{0\}$. It can be further seen that $\hat \Dc$ is dissipative and $\hat \Lc$ is nonnegative. Further, define
\[\mathcal V\coloneqq \K^{2n}\widehat\ominus(\{0\}\times X)\widehat\ominus(X\times \{0\}).\]
The previous considerations show that both $\hat\Dc$ and $\hat\Lc$ are subsets of $\mathcal V$. Moreover, set $k_X\coloneqq\dim X$ and let $\iota:\mathcal V\rightarrow\K^{2(n-k_X)}=\K^{\dim\mathcal V}$ be a vector space isometry. It follows that
\begin{equation}\label{einzelniso}
    \tilde\Dc\coloneqq\iota(\hat\Dc),\quad \tilde\Lc\coloneqq\iota(\hat\Lc)
\end{equation}
are maximally dissipative and maximally nonnegative linear relations in $\K^{2(n-k_X)}$, respectively, satisfying $\mul\tilde\Dc\cap\ker\tilde \Lc=\{0\}$ and note that 
\begin{equation}\label{productiso}
\tilde\Dc\tilde\Lc=\iota(\hat\Dc\hat\Lc).
\end{equation}
Then Proposition~\ref{lem:product} implies the existence of some unitary  $\tilde U\in\K^{(n-k_X)\times(n-k_X)}$, such that, with $k_1\coloneqq\dim(\ran\tilde\Lc\cap\dom\tilde\Dc)$, $k_2\coloneqq n-k_X-k_1$,
\begin{align}
\label{dl_zerl_proof}
\tilde\Dc\tilde\Lc=\ran\diag(\tilde U,\tilde U)\begin{smallbmatrix}
\tilde L_{11}&0\\\tilde L_{21}&\tilde L_{22}\\\tilde D_{11}& 0\\\tilde D_{21}&\tilde D_{22}
\end{smallbmatrix}
\end{align}
for some matrices $\tilde L_{ij},\tilde D_{ij}\in\K^{k_i\times k_j}$ with $\tilde L_{11}\geq 0$, $\ker \tilde L_{11}\subseteq \ker \tilde L_{21}$, $\tilde D_{11}+\tilde D_{11}^*\leq 0$ and
\begin{align}
\label{d22l22proof}
\tilde D_{22}^2=-\tilde D_{22}=-\tilde D_{22}^*,\quad \tilde L_{22}^2=\tilde L_{22}=\tilde L_{22}^*,\quad \ker\tilde D_{22}\widehat\dotplus\ker\tilde L_{22}=\K^{k_2}.
\end{align}
Invoking \eqref{tildeohnetilde}--\eqref{dl_zerl_proof} and 
\[\Vc\widehat\oplus(\{0\}\times X)\widehat\oplus(X\times \{0\})\cong \K^{2(n-k_X)}\times\K^{k_X}\times\K^{k_X} \]
yields the existence of a~unitary matrix $\hat U\in\K^{n\times n}$ such that
\begin{align*}
\Dc\Lc=\diag(\hat U,\hat U)\;\ran \begin{smallbmatrix}\tilde L_{11}&0&0&0\\ \tilde L_{21}&\tilde L_{22}&0&0\\0&0&I_{k_X}&0\\\tilde D_{11}&0&0&0\\ \tilde D_{21}&\tilde D_{22}&0&0\\0&0&0&I_{k_X}\end{smallbmatrix}.
\end{align*}
Lemma~\ref{lem:posreal} (b) implies that $s\tilde L_{11}-\tilde D_{11}$ has only column and row minimal indices equal to zero and their number coincides. 
Hence there exist invertible $S_1,T_1\in\K^{k_1\times k_1}$ and some $n_1\in\N$, such that
\begin{align*}
S_1(s \tilde L_{11}- \tilde D_{11})T_1=\begin{bmatrix}s  L_{11}-  D_{11}&0\\0&0\end{bmatrix}
\end{align*}
for some positive real and regular pencil $s  L_{11}-  D_{11}\in\K[s]^{n_1\times n_1}$. Since $\tilde\Lc$ is maximally nonnegative, Proposition~\ref{lem:product}~(i) yields
\[
\ker L_{11}\times \K^{k_1-n_1}=\ker \begin{smallbmatrix}
L_{11}&0\\0&0
\end{smallbmatrix}=\ker \tilde L_{11}T_1=T_1^{-1}\ker \tilde L_{11}\subset T_1^{-1}\ker \tilde L_{21}=\ker \tilde L_{21}T_1.\]
Consequently, for some $L_{21}^{(1)}\in\K^{k_2\times n_1}$ 
\[\tilde L_{21} T_1=\left[L_{21}^{(1)},0_{k_2\times (k_1-n_1)}\right]\quad \text{and}\quad
\ker  L_{11}\subseteq \ker L_{21}^{(1)}. 
\]
Further, by using $[D_{21}^{(1)},D_{21}^{(2)}]:=\tilde D_{21}T_1$,  $D_{21}^{(1)}\in\K^{k_2\times n_1}$, $D_{21}^{(2)}\in\K^{k_2\times (k_1-n_1)}$, we find
\begin{multline*}
\allowdisplaybreaks
\begin{smallbmatrix}S_1&0&0&0\\0&I_{k_2+k_X}&0&0\\0&0&S_1&0\\0&0&0&I_{k_2+k_X}\end{smallbmatrix}
\;\ran \begin{smallbmatrix}\tilde L_{11}&0&0&0\\ \tilde L_{21}&\tilde L_{22}&0&0\\0&0&I_{k_X}&0\\\tilde D_{11}&0&0&0\\ \tilde D_{21}&\tilde D_{22}&0&0\\0&0&0&I_{k_X}\end{smallbmatrix}
\\=
\ran \begin{smallbmatrix}S_1\tilde L_{11}T_1&0&0&0\\ \tilde L_{21}T_1&\tilde L_{22}&0&0\\0&0&I_{k_X}&0\\ S_1\tilde D_{11}T_1&0&0&0\\ \tilde D_{21}T_1&\tilde D_{22}&0&0\\0&0&0&I_{k_X}\end{smallbmatrix}
=\ran \begin{smallbmatrix}
 L_{11}&0 &0&0&0\\ 0&0&0&0&0\\L_{21}^{(1)}&0& \tilde L_{22}&0&0\\ 0&0&0&I_{k_X}&0\\ D_{11}&0&0&0&0\\0&0&0&0&0\\ D_{21}^{(1)}&D_{21}^{(2)} & \tilde D_{22}&0&0\\0&0&0&0& I_{k_X}
\end{smallbmatrix}.
\end{multline*}
Denoting $k_3\coloneqq\dim\ker \tilde D_{22}$, $n_3\coloneqq\dim\ker \tilde L_{22}$, \eqref{d22l22proof} implies that $k_2=k_3+n_3$. 
Let $\tilde S\in\K^{k_2\times k_2}$ be a~matrix whose first $k_3$ columns form a~basis of $\tilde D_{22}$ and whose last $n_3$ columns form a~basis of $\tilde L_{22}$. Then $\tilde S^*(s\tilde L_{22}-\tilde D_{22})\tilde S=\diag(s \hat L_{22},\hat D_{22})$ for some $\hat L_{22}\in\K^{k_3\times k_3}$, $\hat D_{22}\in\K^{n_3\times n_3}$, which are positive definite by \eqref{d22l22proof}. Then, by taking a~suitable block congruence transformation, we obtain that 
there exists some invertible $S_2\in\K^{k_2\times k_2}$ such that the Weierstra\ss\ form is given by
\begin{align*}
S_2(s\tilde L_{22}-\tilde D_{22})S_2^*=\begin{bmatrix}sI_{k_3}&0\\0&-I_{n_3}\end{bmatrix},
\end{align*}
Hence, with $\begin{smallbmatrix}L_{21}^{(1,1)}\\ L_{21}\end{smallbmatrix}:=S_2L_{21}^{(1)}$ for some $L_{21}^{(1,1)}\in \K^{n_3\times n_1}$ and $L_{21}\in\K^{n_3\times n_1}$ which implies
\[
\ker  L_{11}\subseteq \ker L_{21}^{(1)}=\ker S_2L_{21}^{(1)} \subseteq\ker L_{21}.
\]
Further, decomposing
\[
[S_2 D_{21}^{(1)}, S_2 D_{21}^{(2)}]=\begin{smallbmatrix}D_{21}^{(1,1)}&D_{21}^{(2,1)}\\ D_{21}^{(1,2)} & D_{21}^{(2,2)}\end{smallbmatrix}\in\K^{(k_3+n_3)\times (n_1+(k_1-n_1))}
\]
leads to
\begin{multline}\label{multlinerevival}
\allowdisplaybreaks\begin{smallbmatrix}I_{k_1}&0&0&0&0&0\\0&S_2&0&0&0&0\\ 0&0&I_{k_X}&0&0&0\\0&0&0&I_{k_1}&0&0\\0&0&0&0&S_2&0\\0&0&0&0&0&I_{k_X}\end{smallbmatrix}
\;\ran \begin{smallbmatrix}
 L_{11}&0 &0&0&0\\ 0&0&0&0&0\\L_{21}^{(1)}&0& \tilde L_{22}&0&0\\ 0&0&0&I_{k_X}&0\\  D_{11}&0&0&0&0\\0&0&0&0&0\\ D_{21}^{(1)}&D_{21}^{(2)} & \tilde D_{22}&0&0\\0&0&0&0& I_{k_X}
\end{smallbmatrix}\\
=\ran \begin{smallbmatrix}
 L_{11}&0 &0&0&0\\ 0&0&0&0&0\\S_2L_{21}^{(1)}&0& S_2\tilde L_{22}T_2&0&0\\ 0&0&0&I_{k_X}&0\\  D_{11}&0&0&0&0\\0&0&0&0&0\\ S_2D_{21}^{(1)}&S_2D_{21}^{(2)} & S_2\tilde D_{22}T_2&0&0\\0&0&0&0& I_{k_X}
\end{smallbmatrix}
=\ran \begin{smallbmatrix}
 L_{11}&0 &0&0&0&0\\ 0&0&0&0&0&0\\0&0& I_{k_3}&0&0&0\\ L_{21}&0&0&0&0&0\\ 0&0&0&0&I_{k_X}&0\\  D_{11}&0&0&0&0&0\\0&0&0&0&0&0\\ D_{21}^{(1,1)}&D_{21}^{(2,1)}&0&0&0&0\\ 0&0 &0& -I_{n_3}&0&0\\ 0&0&0&0&0&I_{k_X}
\end{smallbmatrix}
\end{multline}
Now let $S_3\in\K^{k_3\times k_3}$, $T_3\in\K^{(k_1-n_1)\times (k_1-n_1)}$ be invertible with
$S_3D_{21}^{(2,1)}T_3=\begin{smallbmatrix}I_{k_5}&0\\0&0\end{smallbmatrix}$.\\
and $n_2:=k_3-k_5$, then using
\[
\begin{smallbmatrix}D_{21}^{(1,1,1)}\\-D_{21}\end{smallbmatrix}:=S_3D_{21}^{(1,1)},\quad D_{21}^{(1,1,1)}\in\K^{k_5\times n_1},\,D_{21}\in\K^{n_2\times n_1},
\]
we find for the lower five block rows in \eqref{multlinerevival}
\begin{multline}
\allowdisplaybreaks
\label{d_final} 
\begin{smallbmatrix}
I_{n_1}&0&0&0&0\\0&I_{n_1-k_1}&0&0&0\\0&0&S_3&0&0\\0&0&0&I_{n_3}&0\\0&0&0&0&I_{k_X}
\end{smallbmatrix}
\ran\begin{smallbmatrix}  D_{11}&0&0&0&0&0\\0&0&0&0&0&0\\ D_{21}^{(1,1)}&D_{21}^{(2,1)}&0&0&0&0\\ 0&0 &0& -I_{n_3}&0&0\\ 0&0&0&0&0&I_{k_X}\end{smallbmatrix}\\=
\ran \begin{smallbmatrix}
D_{11}&0&0&0&0&0&0\\0&0&0&0&0&0&0\\ D_{21}^{(1,1,1)}&I_{k_5}&0&0&0&0&0\\ D_{21}^{(1,1,2)}&0&0&0&0&0&0 \\ 0&0 &0&0& -I_{n_3}&0&0\\ 0&0&0&0&0&0&I_{k_X}
\end{smallbmatrix}=\ran \begin{smallbmatrix}
D_{11}&0&0&0&0&0&0&0\\0&0&0&0&0&0&0&0\\ 0&I_{k_5}&0&0&0&0&0&0\\ - D_{21}&0&0&0&0&0&0&0 \\ 0&0 &0&0&0& -I_{n_3}&0&0\\ 0&0&0&0&0&0&0& I_{k_X}
\end{smallbmatrix}
\end{multline}
and for the upper five block rows in \eqref{multlinerevival}
\begin{align}
\label{l_final}
&\begin{smallbmatrix}
I_{n_1}&0&0&0&0\\0&I_{n_1-k_1}&0&0&0\\0&0&S_3&0&0\\0&0&0&I_{k_4}&0\\0&0&0&0&I_{k_X}
\end{smallbmatrix}
\;\ran \begin{smallbmatrix}L_{11}&0 &0&0&0&0\\ 0&0&0&0&0&0\\0&0& I_{k_3}&0&0&0\\ L_{21}&0&0&0&0&0\\ 0&0&0&0&I_{k_X}&0
\end{smallbmatrix}
\\&=\ran \begin{smallbmatrix}
 L_{11}&0 &0&0&0&0\\ 0&0&0&0&0&0\\0&0& I_{k_3}&0&0&0\\ L_{21}&0&0&0&0&0\\0&0&0&0&I_{k_X}&0
\end{smallbmatrix}=\ran \begin{smallbmatrix}
 L_{11}&0&0 &0&0&0&0&0\\ 0&0&0&0&0&0&0&0\\0&0&0&I_{k_5}&0&0&0&0\\ 0&0&0&0&I_{n_2}&0&0&0\\L_{21}&0&0&0&0&0&0&0\\0&0&0&0&0&0&I_{k_X}&0
\end{smallbmatrix}.\nonumber
\end{align}
Then the form \eqref{finalpencil} is achieved by setting $n_4:=k_5+k_X$ and performing a~joint permutation of block rows of the form $2\rightarrow6\rightarrow5\rightarrow3\rightarrow4\rightarrow2$ and block columns ($3\rightarrow8\rightarrow7\rightarrow5\rightarrow2\rightarrow6\rightarrow3$) of the matrices on the right hand side in \eqref{d_final}  and \eqref{l_final}. Combining all of the so far transformations leads to an invertible  $S\in\K^{n\times n}$ with \eqref{finalpencil}.\\
\noindent
{\em Step 2:} 
Let $E,A\in\K^{n\times m}$ be such that  $\ran\begin{smallbmatrix}
E\\A
\end{smallbmatrix}=\Dc\Lc$ for some maximally dissipative relation $\Dc\subseteq\K^{2n}$ and some maximally nonnegative relation $\Lc\subseteq\K^{2n}$. Then the result from Step~1 gives
\begin{align}
\label{rangeequal}
\ran\begin{smallbmatrix}E\\A\end{smallbmatrix}=\Dc\Lc=\diag(S^{-1},S^{-1})\ran\begin{smallbmatrix}
L\\  D
\end{smallbmatrix}
\end{align}
with matrices $L,D\in\K^{n\times\hat m}$ as in \eqref{finalpencil}.
If $m\geq \hat m$ then there exists some invertible $T\in\K^{m\times m}$ such that
\[
\begin{smallbmatrix}
SE\\SA
\end{smallbmatrix}T=\begin{smallbmatrix}
 L&0\\ D&0
\end{smallbmatrix}
.\]
Hence \eqref{maximalkcf_hinten} follows from \eqref{finalpencil}.
If $m<\hat m$ then the block structure in \eqref{finalpencil} implies that  $d\coloneqq \dim\Dc\Lc=\dim\ran\begin{smallbmatrix}
 L\\  D
\end{smallbmatrix}=n_1+n_2+n_3+2n_4$ and that the first $d$ columns in  $\begin{smallbmatrix}
 L\\  D
\end{smallbmatrix}$ are linearly independent. Since $m\geq d$, we can remove $\hat m-m$ zero columns from $ L$ and $ D$ which leads to matrices $\hat L, \hat D\in\K^{n\times m}$ which are still of the form \eqref{finalpencil}.  Observe that \eqref{rangeequal} still holds after replacing $L$ with $\hat L$ and $D$ with $\hat D$. Hence there exists some invertible $T\in\K^{m\times m}$ such that $S(sE-A)T=s \hat L- \hat D$ which implies \eqref{maximalkcf_hinten}.\hfill
\end{proof}

\bibliographystyle{siam}


\bibliography{references}


\end{document}